\documentclass[12pt,a4paper]{amsart}

\usepackage{amsmath, nicefrac, amsthm, verbatim, amsfonts, amssymb, xcolor, bbm,mathtools}
\usepackage{graphics, xspace, enumerate}
\usepackage[a4paper,margin=3cm]{geometry}

\usepackage{dsfont}

\usepackage[bb=boondox]{mathalfa}

\usepackage{graphicx,amsmath}
\usepackage[colorlinks=true,citecolor=green,urlcolor=green,linkcolor=green,bookmarksopen=true,unicode=true,pdffitwindow=true]{hyperref}
\usepackage[english]{babel}
\usepackage[languagenames,fixlanguage]{babelbib}
\usepackage{seqsplit,cleveref}
\hypersetup{pdfauthor={}}
\hypersetup{pdftitle={Periodic Homogenization of a L\'evy-type Process with Small Jumps}}

\theoremstyle{plain}
\newtheorem{theorem}{Theorem}[section]

\newtheorem{lemma}[theorem]{Lemma}
\newtheorem{proposition}[theorem]{Proposition}
\theoremstyle{definition}
\newtheorem{remark}[theorem]{Remark}

\newtheorem{definition}[theorem]{Definition}
\newcommand{\E}{\mathrm{e}}

\newcommand {\R} {\ensuremath{\mathbb{R}}}
\newcommand {\ZZ} {\ensuremath{\mathbb{Z}}}
\newcommand {\N} {\ensuremath{\mathbb{N}}}
\newcommand {\CC} {\ensuremath{\mathbb{C}}}
\newcommand{\process}[1]{\{#1_t\}_{t\geq0}}

\newcommand{\D}{\mathrm{d}}
\newcommand{\Ind}{\mathds{1}}

\numberwithin{equation}{section}

\title[Periodic Homogenization of  a
L\'evy-type Process
with Small Jumps]{Periodic Homogenization of  a
L\'evy-type Process
with Small Jumps}

\author[N.\ Sandri\'{c}]{Nikola Sandri\'{c}}
\address[Nikola\ Sandri\'{c}]{Department of Mathematics\\University of Zagreb\\ Zagreb\\Croatia}
\email{nsandric@math.hr}

\author[I.\ Valenti\'{c}]{Ivana Valenti\'{c}}
\address[Ivana\ Valenti\'{c}]{Department of Mathematics\\University of Zagreb\\ Zagreb\\Croatia}
\email{ivana.valentic@math.hr}

\author[J.\ Wang]{Jian Wang}
\address[Jian\ Wang]{College of Mathematics and Informatics  \&
Fujian Key Laboratory of Mathematical Analysis and Applications (FJKLMAA)  \& Center for Applied Mathematics of Fujian Province (FJNU)\\ Fujian Normal University\\ Fuzhou\\ PR China}
\email{jianwang@fjnu.edu.cn}


\begin{document}
\allowdisplaybreaks[4]

\begin{abstract}
    In this article, we consider the problem of periodic homogenization of  a    Feller process generated by a
    pseudo-differential operator, the so-called L\'evy-type process.
    Under the assumptions that the  generator  has     rapidly periodically oscillating   coefficients, and that it admits  ``small  jumps'' only (that is, the jump kernel has  finite second moment), we prove that the appropriately centered and scaled process converges weakly to a Brownian motion with covariance matrix given in terms of the coefficients of the generator.
    The presented results generalize the  classical and well-known results related to
    periodic homogenization of a  diffusion process.

\end{abstract}

\maketitle

\section{Introduction}\label{S1}

 The classical reaction-diffusion  equation
\begin{equation*}\partial_t p(t,x)\,=\,\langle b(x),\nabla_xp(t,x)\rangle +
\frac{1}{2}{\rm Tr}\,c(x)\nabla_x^{2}p(t,x)
+r\bigl(p(t,x)\bigr)\end{equation*}  describes the evolution of population density due to random displacement of individuals (diffusion term), movement of individuals within the environment (drift term), and their reproduction (reaction term).
In order to characterize long-range effects the diffusion and drift terms  are
naturally replaced by
 an integro-differential operator  of the following form \begin{equation}\begin{aligned}\label{IDO} \mathcal{L}f(x)\,=\, & \langle b(x),\nabla f(x)\rangle+
	\frac{1}{2}{\rm Tr}\,c(x)\nabla^{2}f(x)
	\\&+\int_{\R^{d}}
\left(f(x+y)-f(x)-\langle y,\nabla f(x)\rangle \Ind_{B_1(0)}(y)\right)\,\nu(x,\D y)\,,\end{aligned}\end{equation}
where $\nu(x,\D y)$ is
a non-negative Borel kernel  which describes these effects, that is, it  quantifies the property that an individual at  $x$ jumps to  $x+\D y$.

The main goal of this article is to discuss periodic homogenization of the operator $\mathcal{L}$, with kernel $\nu(x,\D y)$ admitting ``small jumps'' only (that is, having finite second moment). Our approach is
 based on probabilistic techniques. More precisely, we discuss periodic  homogenization of the stochastic (Markov) process $\process{X}$ in
periodic medium, generated by $\mathcal{L}$.  We focus to the case when $\process{X}$ is a so-called L\'evy-type process or, equivalently, when $\mathcal{L}$ is a pseudo-differential operator (see below for details).
Roughly speaking, we show that
the
appropriately centered and scaled process $\process{X}$: \begin{equation}\label{e:eeeff}\{\varepsilon X_{\varepsilon^{-2}t}-\varepsilon^{-1}\bar {b^*} t\}_{t\ge0},\end{equation}
for some $\bar {b^*}\in \R^d$,
converges, as $\varepsilon\to0$,  in the path space endowed  with the Skorohod ${\rm J}_1$-topology to a $d$-dimensional zero-drift Brownian motion determined by covariance matrix of the form   \begin{equation}\begin{aligned}\label{ET1.2}\Sigma\,:=\,\Bigg(&\int_{\mathbb{T}_\tau^d}\sum_{k,l=1}^{d}\left(\delta_{ki}-\partial_k\beta_i(x)\right)c_{kl}(x)\left(\delta_{lj}-\partial_l\beta_j(x)\right)\pi(\D x)\\
&+\int_{\mathbb{T}_\tau^d}\int_{\R^{d}}y_iy_j\,\nu(x,\D y)\,\pi(\D x)\\
&+\int_{\mathbb{T}_\tau^d}\int_{\R^d}\bigl(\beta_i(x+y)-\beta_i(x)\bigr)\bigl(\beta_j(x+y)-\beta_j(x)\bigr)\nu(x,\D y)\,\pi(\D x)\\& -2\int_{\mathbb{T}_\tau^d}\int_{\R^d}y_i\bigl(\beta_j(x+y)-\beta_j(x)\bigr)\nu(x,\D y)\,\pi(\D x)
\Bigg)_{1\leq i,j\leq d}\,,\end{aligned}\end{equation}
(see \Cref{T1.1} for details). Equivalently, according to \cite[Theorem 7.1]{rene-bjorn-jian}, $$\lim_{\varepsilon\to0}\lVert\mathcal{L}_{\varepsilon}f-\varepsilon^{-1}\langle\bar {b^*},\nabla f\rangle-2^{-1}{\rm Tr}\,\Sigma \nabla^2f\rVert_\infty\,=\,0\,,\qquad f\in C_c^\infty(\R^d)\,,$$ where
\begin{align*} \mathcal{L}_\varepsilon f(x)\,=\, & \varepsilon^{-1}\langle b(x/\varepsilon),\nabla f(x)\rangle+
\frac{1}{2}{\rm Tr}\,c(x/\varepsilon)\nabla^{2}f(x)
\\&+\varepsilon^{-2}\int_{\R^{d}}
\left(f(x+\varepsilon y)-f(x)-\varepsilon \langle y,\nabla f(x)\rangle  \Ind_{B_1(0)}(y)\right)\,\nu(x/\varepsilon,\D y)\,.\end{align*}
 Let us remark that when $b(x)\equiv0$ and $\nu(x,\D y)$ is symmetric for all $x\in\R^d$, centralization in \cref{e:eeeff} is not necessary (that is, one can take $\bar {b^*}=0$), and $\beta(x)\equiv 0$ in \cref{ET1.2}. Thus, in this case, $\Sigma$  is reduced to
\begin{equation*}\begin{aligned} \Bigg(&\int_{\mathbb{T}_\tau^d} c_{ij}(x) \,\pi(\D x)+\int_{\mathbb{T}_\tau^d}\int_{\R^{d}}y_iy_j\,\nu(x,\D y)\,\pi(\D x)
\Bigg)_{1\leq i,j\leq d}\,,\end{aligned}\end{equation*} (see \cite{hom} for more details).

\subsection*{Preliminaries on L\'evy-Type Processes}\label{SS1.1}
 Let
$(\Omega,\mathcal{F},\{\mathbb{P}_{x}\}_{x\in\R^{d}},\process{\mathcal{F}},\process{\theta},$ \linebreak $\process{X})$, denoted by $\process{X}$
in the sequel, be a Markov process on  state space
$(\R^{d},\mathcal{B}(\R^{d}))$ (see \cite{BG-68}). Here, $d\geq1$, and
$\mathcal{B}(\R^{d})$ denotes the Borel $\sigma$-algebra on
$\R^{d}$. Due to the Markov property, the associated family of linear operators $\process{P}$ on
$B_b(\R^{d})$ (the space of bounded and Borel measurable functions),
defined by $$P_tf(x)\,:=\, \mathbb{E}_{x}\bigl[f(X_t)\bigr]\,,\qquad t\geq0\,,\
x\in\R^{d}\,,\ f\in B_b(\R^{d})\,,$$  forms a \emph{semigroup}  on the
Banach space $(B_b(\R^{d}),\lVert\cdot\rVert_\infty)$, that is, $P_0={\rm Id}$ and $P_s\circ
P_t=P_{s+t}$   for all $s,t\geq0$. Here, $\mathbb{E}_x$ stands for the expectation with respect to $\mathbb{P}_x(\D\omega)$, $x\in\R^{d}$, and
$\lVert\cdot\rVert_\infty$ and ${\rm Id}$ denote the supremum norm and the identity operator, respectively, on the space
$B_b(\R^{d})$. Moreover, the semigroup $\process{P}$ is
\emph{contractive} ($\lVert P_tf\rVert_{\infty}\leq\lVert f\rVert_{\infty}$
for all $t\geq0$ and  $f\in B_b(\R^{d})$) and \emph{positivity
	preserving} ($P_tf\geq 0$ for all $t\geq0$ and  $f\in
B_b(\R^{d})$ satisfying $f\geq0$). The \emph{infinitesimal generator}
$(\mathcal{A}^{b},\mathcal{D}_{\mathcal{A}^{b}})$ of the semigroup
$\process{P}$ (or of the process $\process{X}$) is a linear operator
$\mathcal{A}^{b}:\mathcal{D}_{\mathcal{A}^{b}}\to B_b(\R^{d})$
defined by
$$\mathcal{A}^{b}f\,:=\,
\lim_{t\to0}\frac{P_tf-f}{t},\qquad f\in\mathcal{D}_{\mathcal{A}^{b}}\,:=\,\left\{f\in B_b(\R^{d}):
\lim_{t\to0}\frac{P_t f-f}{t} \ \textrm{exists in}\
\lVert\cdot\rVert_\infty\right\}\,.
$$ We call $(\mathcal{A}^{b},\mathcal{D}_{\mathcal{A}^{b}})$ the \emph{$B_b$-generator} for short.
A Markov process $\process{X}$ is said to be a \emph{Feller process}
if its corresponding  semigroup $\process{P}$ forms a \emph{Feller
	semigroup}. This means that

\medskip

\begin{itemize}
	\item [(i)] $\process{P}$ enjoys the \emph{Feller property}, that is,  $P_t(C_\infty(\R^{d}))\subseteq C_\infty(\R^{d})$ for all $t\geq0$;
	
	\medskip
	
	\item [(ii)] $\process{P}$ is \emph{strongly continuous}, that is, $\lim_{t\to0}\lVert P_tf-f\rVert_{\infty}=0$ for all $f\in
	C_\infty(\R^{d})$.
\end{itemize}

\medskip

\noindent Here, $C_\infty(\R^{d})$ denotes
the space of continuous functions vanishing at infinity. Recall also that a Markov process $\process{X}$ is said to be a $C_b$-\emph{Feller}
(resp.\ \textit{strong Feller}) process if the corresponding  semigroup $\process{P}$ satisfies $P_tf\in C_b(\R^d)$ for all $t>0$ and all $f\in C_b(\R^d)$
(resp.\ $f\in B_b(\R^d)$), where $C_b(\R^d):=C(\R^d)\cap B_b(\R^d)$.
 Note that
every Feller semigroup $\process{P}$  can be uniquely extended to
$B_b(\R^{d})$ (see \cite[Section 3]{rene-conserv}). For notational
simplicity, we denote this extension  by $\process{P}$ again. Also,
let us remark that every Feller process (admits a modification that)
has  c\`adl\`ag sample paths and
possesses the strong Markov
property (see  \cite[Theorems 3.4.19 and
3.5.14]{jacobIII}).   Further,
in the case of Feller processes, we call
$(\mathcal{A}^{\infty},\mathcal{D}_{\mathcal{A}^{\infty}}):=(\mathcal{A}^{b},\mathcal{D}_{\mathcal{A}^{b}}\cap
C_\infty(\R^{d}))$ the \emph{Feller generator} for short. Observe
that in this case $\mathcal{D}_{\mathcal{A}^{\infty}}\subseteq
C_\infty(\R^{d})$ and
$\mathcal{A}^{\infty}(\mathcal{D}_{\mathcal{A}^{\infty}})\subseteq
C_\infty(\R^{d})$. If the set of smooth functions
with compact support $C_c^{\infty}(\R^{d})$ is contained in
$\mathcal{D}_{\mathcal{A}^{\infty}}$, that is, if the Feller
generator
$(\mathcal{A}^{\infty},\mathcal{D}_{\mathcal{A}^{\infty}})$ of the
Feller process $\process{X}$ satisfies

\medskip

\begin{description}
	\item[(\textbf{LTP1})]
	$C_c^{\infty}(\R^{d})\subseteq\mathcal{D}_{\mathcal{A}^{\infty}}$,
\end{description}

\medskip

\noindent then, according to \cite[Theorem 3.4]{courrege-symbol},
$\mathcal{A}^{\infty}|_{C_c^{\infty}(\R^{d})}$ is a \textit{pseudo-differential operator}, that is, it can be written in the form
\begin{equation}\label{PDO}\mathcal{A}^{\infty}|_{C_c^{\infty}(\R^{d})}f(x) \,=\, -\int_{\R^{d}}q(x,\xi)\E^{i\langle \xi,x\rangle}
\hat{f}(\xi)\, \D\xi\,,\end{equation}  where $\hat{f}(\xi):=
(2\pi)^{-d} \int_{\R^{d}} \E^{-i\langle\xi,x\rangle} f(x)\, \D x$ denotes
the Fourier transform of the function $f(x)$. The function $q :
\R^{d}\times \R^{d}\to \CC$ is called  the \emph{symbol}
of the pseudo-differential operator. It is measurable and locally
bounded in $(x,\xi)$, and is continuous and negative definite as a
function of $\xi$. Hence, by \cite[Theorem 3.7.7]{jacobI}, the
function $\xi\mapsto q(x,\xi)$ has for each $x\in\R^{d}$ the
following L\'{e}vy-Khintchine representation \begin{equation}\begin{aligned}\label{SIMB}q(x,\xi) \,=\,&a(x)-
i
\langle \xi, b(x)\rangle + \frac{1}{2}\langle\xi,c(x)\xi\rangle\\& +
\int_{\R^{d}}\left(1-\E^{i\langle\xi,y\rangle}+i\langle\xi, y\rangle\Ind_{B_1(0)}(y)\right)\nu(x,\D y)\,,\end{aligned}\end{equation}
where  $a(x)$ is a non-negative Borel measurable function, $b(x)$ is
an $\R^{d}$-valued Borel measurable function,
$c(x):=(c_{ij}(x))_{1\leq i,j\leq d}$ is a symmetric non-negative
definite $d\times d$ matrix-valued Borel measurable function,
and $\nu(x,\D y)$ is a non-negative Borel kernel on $\R^{d}\times
\mathcal{B}(\R^{d})$, called the \emph{L\'evy
kernel}, satisfying
$$\nu(x,\{0\})\,=\,0\qquad \textrm{and} \qquad \int_{\R^{d}}\left(1\wedge
|y|^{2}\right)\nu(x,\D y)\,<\,\infty,\qquad x\in\R^{d}\,.$$
The quadruple
$(a(x),b(x),c(x),\nu(x,\D y))$ is called the \emph{L\'{e}vy quadruple}
of
$\mathcal{A}^{\infty}|_{C_c^{\infty}(\R^{d})}$ (or of  $q(x,\xi)$).
Let us remark that  local boundedness of $q(x,\xi)$  implies  local boundedness of the corresponding $x$-coefficients, and \textit{vice versa}
(see \cite[Lemma 2.1 and Remark 2.2]{rene-holder}).
In the sequel, we assume the following condition on the symbol
$q(x,\xi)$:

\medskip

\begin{description}
	\item[(\textbf{LTP2})] $q(x,0)=a(x)\equiv0$.
\end{description}

\medskip

\noindent This condition is closely related to the conservativeness property of $\process{X}$.
Namely, under  the assumption that the
$x$-coefficients
 of $q(x,\xi)$ are uniformly bounded (which is certainly the case in the periodic setting), (\textbf{LTP2}) implies that $\process{X}$ is \emph{conservative}, that is, $\mathbb{P}_{x}(X_t\in\R^{d})=1$ for all $t\geq0$ and
$x\in\R^{d}$ (see \cite[Theorem 5.2]{rene-conserv}). Further, note that by combining \cref{PDO,SIMB}
 with (\textbf{LTP2}),  $\mathcal{A}^{\infty}|_{C_c^{\infty}(\R^{d})}$ takes the form \cref{IDO}.
Conversely, if $\mathcal{L}: C_c^\infty(\R^d)\to C_\infty(\R^d)$ is a linear  operator  of the form    \cref{IDO} satisfying the so-called \textit{positive maximum principle} ($\mathcal{L}f(x_0)\le0$ for any $f\in C_c^\infty(\R^d)$ with $f(x_0)=\sup_{x\in\R^d}f(x)\ge0$) and such that $\bigl(\lambda-\mathcal{L}\bigr)(C_c^\infty(\R^d))$ is dense in $C_\infty(\R^d)$ for some (or all) $\lambda>0$, then, according to
 the
 Hille-Yosida-Ray theorem, $\mathcal{L}$ is closable and the closure  is the generator of a Feller semigroup. In particular, the corresponding Feller process is a L\'evy-type process.
In the case when  $q(x,\xi)$ does not depend
on the variable $x\in\R^{d}$, $\process{X}$ becomes a \emph{L\'evy
	process}, that is, a stochastic process   with stationary and
independent increments. Moreover, unlike Feller
processes, every L\'evy process is uniquely and completely
characterized through its corresponding symbol (see \cite[Theorems 7.10 and 8.1]{sato-book} and \cite[Example 2.26]{rene-bjorn-jian}). According to this, it is not hard to
check that every
conservative L\'evy process satisfies conditions
(\textbf{LTP1}) and (\textbf{LTP2})  (see \cite[Theorem 31.5]{sato-book}).
Thus, the class of processes we consider in this article contains
L\'evy processes. Throughout this article, the symbol $\process{X}$ denotes a Feller
process satisfying conditions (\textbf{LTP1}) and (\textbf{LTP2}). Such a
process is called a \emph{L\'evy-type process} (LTP).
If $\nu(x,\D y)\equiv0$,  $\process{X}$ is  called a \emph{diffusion process}.  Note that this definition agrees with the standard definition of (Feller-Dynkin) diffusions (see \cite[Chapter III.2]{rogersI}). A typical example of a LTP is a solution to the following SDE
\begin{equation}\label{SDE1}
\D X_t\,=\,\Phi(X_{t-})\,\D Y_t\,,\qquad  X_0=x\in\R^d\,,
\end{equation} where $\Phi:\R^{d}\to\R^{d\times n}$ is locally Lipschitz continuous and bounded (which is not a restriction in the periodic setting), and $\process{Y}$ is an $n$-dimensional L\'evy process with symbol $q_Y(\xi)$. Namely, in \cite[Theorems 3.1 and 3.5 and Corollary 3.3]{schnurr} it has been shown that the unique solution $\process{X}$ to the SDE in \cref{SDE1} (which exists by standard arguments) is a LTP with symbol of the form $q(x,\xi)=q_Y(\Phi'(x)\xi).$ Here, for a  matrix $M$, $M'$ denotes its transpose.  Observe that the following SDE is a special case of \cref{SDE1},
\begin{equation}\label{SDE2}
\D X_t\,=\,\Phi_1(X_t
)\,\D t+\Phi_2(X_t
)\,\D B_t\,+\Phi_3(X_{t-})\,\D Z_t\,,\qquad  X_0=x\in\R^d\,,
\end{equation}
where  $\Phi_1:\R^{d}\to\R^{d}$, $\Phi_2:\R^{d}\to\R^{d\times p}$ and $\Phi_3:\R^{d}\to\R^{d\times q}$, with $p+q=n-1$, are locally Lipschitz continuous and bounded, $\process{B}$ is a $p$-dimensional Brownian motion, and $\process{Z}$ is a $q$-dimensional pure-jump L\'evy process (that is, a L\'evy process determined by a L\'evy triplet of the form $(0,0,\nu_Z(\D y))$). Namely, set
$\Phi(x)=\bigl(\Phi_1(x),\Phi_2(x),\Phi_3(x)\bigr)$
 for any $x\in \R^d$, and $Y_t=(t,B_t,Z_t)'$
 for $t\ge0$. For more
on L\'evy-type processes  we refer the readers to the monograph
\cite{rene-bjorn-jian}.

\subsection*{LTPs with Periodic Coefficients}
Let $\tau=(\tau_1,\ldots, \tau_{d})\in (0,\infty)^{d}$ be
fixed, and let $\tau\ZZ^{d}:=\tau_1\ZZ\times\ldots\times\tau_{d} \ZZ.$
For $k=(k_1,\dots,k_d)\in\ZZ^d$ define $\tau\odot k:=(\tau_1k_1,\dots,\tau_dk_d)$, and for $x\in\R^{d}$ define
$$x_\tau\,:=\,\{y\in\R^{d}:x-y\in\tau\ZZ^{d}\}\qquad\textrm{and}\qquad
\R^{d}/\tau\ZZ^{d}\,:=\,\{x_\tau:x\in\R^{d}\}\,.$$ In the sequel, we denote $\mathbb{T}^d_\tau=\R^{d}/\tau\ZZ^{d}$. Clearly,
$\mathbb{T}^d_\tau$ is obtained
by identifying the opposite
faces of $[0,\tau]:=[0,\tau_1]\times\ldots\times[0,\tau_{d}]$.
Let
$\Pi_{\tau} : \R^{d}\to \mathbb{T}^d_\tau$, $\Pi_{\tau}(x):=x_\tau$, be the covering map.
A
function $f:\R^{d}\to\R$ is called \textit{$\tau$-periodic} if
$$f(x+\tau\odot k)\,=\,f(x)\,,\qquad x\in\R^{d}\,,\ k\in\ZZ^d\,.$$
Clearly, every $\tau$-periodic function $f(x)$ is completely and uniquely determined by its restriction $f|_{[0,\tau]}(x)$ to $[0,\tau]$, and since  $f|_{[0,\tau]}(x)$ assumes the same value on opposite faces of $[0,\tau]$, it can be identified  by a function $f_\tau:\mathbb{T}^d_\tau\to\R$ given with $f_\tau (x_\tau)=f(x).$
For notational convenience, we will often omit the subscript $\tau$ and simply write $x$ instead of $x_\tau$, and $f(x)$ instead of $f_\tau(x)$.

Let now $\process{X}$ be  a LTP with semigroup $\process{P}$, symbol $q(x,\xi)$ and L\'evy triplet $(b(x),c(x),\nu(x,\D y))$, satisfying:
\medskip

\begin{description}
	\item [(\textbf{C1})]  $\displaystyle x\mapsto q(x,\xi)$  is $\tau$-periodic for all $\xi\in\R^{d}$.
\end{description}

\medskip

 \noindent Directly  from the L\'evy-Khintchine formula it follows that (\textbf{C1}) is equivalent to the $\tau$-periodicity of the corresponding L\'evy triplet $(b(x),c(x),\nu(x,\D y))$, which  in turn is equivalent to the $\tau$-periodicity of $x\mapsto\mathbb{P}_x(X_t-x\in\D y)$ $($see \cite[Section 4]{hom}$)$.
This immediately implies  that $\process{P}$ preserves the class of all bounded Borel measurable $\tau$-periodic functions, that is, the function $x\mapsto P_tf(x)$ is $\tau$-periodic for all $t\geq0$ and all $\tau$-periodic $f\in B_b(\R^{d})$. Now, together with this, a straightforward adaptation of \cite[Proposition 3.8.3]{vasili-book} entails that  $\{\Pi_\tau(X_t)\}_{t\ge0}$ is a
Markov process on $(\mathbb{T}_\tau^d,\mathcal{B}(\mathbb{T}_\tau^d))$ with positivity preserving contraction semigroup $\process{P^{\tau}}$
(on the space $(B_b(\mathbb{T}_\tau^d),\lVert\cdot\rVert_\infty)$) given by
$$P_t^{\tau}f(x)\,:=\,\mathbb{E}^{\tau}_x\bigl[f(\Pi_\tau(X_t))\bigr]\,=\,\int_{\mathbb{T}_\tau^d}f(y)\mathbb{P}_x^{\tau}\bigl(\Pi_\tau(X_t)\in \D y\bigr) \,,$$ for $t\geq0,$ $ x\in\mathbb{T}_\tau^d$ and $f\in B_b(\mathbb{T}_\tau^d)$. Here, $\mathcal{B}(\mathbb{T}_\tau^d)$ stands for the Borel $\sigma$-algebra on $\mathbb{T}_\tau^d$ (with respect to the standard quotient topology),  $B_b(\mathbb{T}_\tau^d)$ denotes the class of all bounded Borel measurable functions $f:\mathbb{T}_\tau^d\to\R$ (which can be identified with the class of all $\tau$-periodic bounded Borel measurable functions $f:\R^d\to\R$), and
\begin{equation*}\label{E2.1}\mathbb{P}_x^{\tau}\bigl(\Pi_\tau(X_t)\in B\bigr)\,:=\,\mathbb{P}_{z_x}\bigl(X_t\in \Pi_\tau^{-1}(B)\bigr)\,,\qquad t\ge0\,,\ x\in\mathbb{T}_\tau^d\,,\ B\in\mathcal{B}(\mathbb{T}_\tau^d)\,,\end{equation*}  with    $z_x$  being an arbitrary point  in
$\Pi^{-1}_{\tau}(\{x\})$.

Further, assume that
\smallskip

\begin{description}
	\item [(\textbf{C2})]  $\process{X}$ is strong Feller and \textit{open-set irreducible}, that is, for any $t>0$, any $x\in \R^d$ and any non-empty open set $O\subseteq\R^d$,
	$\displaystyle\mathbb{P}_{x}(X_t\in O)>0$.
\end{description}
\smallskip

\noindent Clearly, (\textbf{C2})  automatically implies  that the process $\{\Pi_\tau(X_t)\}_{t\ge0}$ is strong Feller and open-set irreducible, too.
Hence, by employing \cite[Remark 3.2]{liang-sch-wang} and \cite[Theorem 1.1]{wat} we have proved the following.

\begin{proposition}\label{SS2.1} The process $\{\Pi_\tau(X_t)\}_{t\ge0}$ admits a unique invariant probability measure $\pi(\D x)$, that is, a measure $\pi(\D x)$ satisfying $$\int_{\mathbb{T}_\tau^d}\mathbb{P}_x^\tau\bigl(\Pi_\tau(X_t)\in B\bigr)\,\pi(\D x)\,=\,\pi(B)\,,\qquad t\ge0\,,\ B\in\mathcal{B}(\mathbb{T}_\tau^d)\,,$$ such that \begin{equation}
	\label{eq:erg}
	\sup_{x\in\mathbb{T}_\tau^d}\lVert \mathbb{P}_x^\tau\bigl(\Pi_\tau(X_t)\in \D y\bigr)-\pi(\D y) \rVert_{{\rm TV}}\,\le\, \Gamma \E^{-\gamma t}\,,\qquad t\ge0\,	\end{equation} for some $\gamma,\Gamma>0$, where $\lVert\cdot\rVert_{{\rm TV}}$ denotes the total variation norm on the space of signed measures on $\mathcal{B}(\mathbb{T}_\tau^d)$.
\end{proposition}

\medskip

\begin{remark}  Alternatively, \Cref{SS2.1} is a consequence of  \cite[Theorems 3.2 and 8.1]{meyn-tweedie-II} and \cite[Theorem 3.2]{tweedie-mproc}, or  \cite[Theorem 6.1]{meynIII} and \cite[Theorem 5.1]{tweedie-mproc} (by setting $V(x)\equiv1$ and $c=d=1$). Also,  if instead of (\textbf{C2}) we assume
	
	\medskip
	
	\begin{description}
		\item [\textbf{($\widetilde {\text{C2}})$}] $\process{X}$ admits a density function $p_t(x,y)$ (with respect to the Lebesgue measure) satisfying
		
		\medskip
		
		\begin{itemize}
			\item[(i)] for any $t>0$, the function $(x,y)\mapsto p_t(x,y)$ is continuous on $\R^d\times \R^d$;

			\medskip
			
			\item[(ii)] there is a non-empty open set $O\subseteq\R^d$  such that $p_t(x,y)>0$ for all $t>0$, $x\in \R^d$ and $y\in O$,
		\end{itemize}
	\end{description}
	
	\medskip
	
	\noindent which guarantees that
D\"oblin's irreducibility condition holds true (see \cite[page 256]{doob}), then \Cref{SS2.1} follows from \cite[Theorem 3.1]{benso-lions-book}.
\end{remark}

Conditions (in terms of the L\'evy triplet $(b(x),c(x),\nu(x,\D y))$) ensuring (\textbf{C2}) are discussed in \Cref{examples}.

\subsection*{The Semimartingale Nature of LTPs}

As we have already commented, the problem of
 homogenization of an operator
of the form \cref{IDO} corresponding to a LTP
   is equivalent to the convergence of the corresponding family of LTPs in the path space endowed with the Skorohod ${\rm J}_1$ topology (see \cite[Theorem 7.1]{rene-bjorn-jian}). According to
    \cite[Lemma 3.2]{rene-holder},
$\{X_t\}_{t\ge0}$ is a
$\mathbb{P}_{x}$- semimartingale (with respect to the natural filtration)
for any $x\in\R^d$. Therefore, in order to show this convergence, our aim is to employ \cite[Theorem VIII.2.17]{jacod} which states that a sequence of semimartingales converges in the path space endowed with the Skorohod ${\rm J}_1$ topology to a process with independent increments if the corresponding semimartingale characteristics converge in probability.

Let us now  recall the notion of characteristics of a semimartingale
(see \cite{jacod}). Let
$(\Omega,\mathcal{F},\process{\mathcal{F}},\mathbb{P},\process{S})$, denoted by
$\process{S}$ in the sequel, be a $d$-dimensional semimartingale, and
let $h:\R^{d}\to\R^{d}$ be a truncation function (that is, a
bounded Borel measurable function which satisfies $h(x)=x$ in a neighborhood of
the origin).
Define
$$\check{S}(h)_t\,:=\,\sum_{s\leq t}\bigl(\Delta S_s-h(\Delta S_s)\bigr)\quad
\textrm{and} \quad S(h)_t\,:=\,S_t-\check{S}(h)_t\,,\qquad t\geq0\,,$$ where the process
$\process{\Delta S}$ is defined by $\Delta S_t:=S_t-S_{t-}$ and
$\Delta S_0:=S_0$. The process $\process{S(h)}$ is a \emph{special
	semimartingale}, that is, it admits a unique decomposition
\begin{equation}\label{SS}S(h)_t\,=\,S_0+M(h)_t+B(h)_t\,,\end{equation} where $\process{M(h)}$ is a local
martingale, and $\process{B(h)}$ is a predictable process of bounded
variation.
\begin{definition}
	Let $\process{S}$  be a
	semimartingale, and let $h:\R^{d}\longrightarrow\R^{d}$ be a truncation
	function. Furthermore, let $\process{B(h)}$  be the predictable
	process
	of bounded variation appearing in \cref{SS},  let $N(\omega,\D y,\D s)$ be the
	compensator of the jump measure
	$$\mu(\omega,\D y,\D s)\,:=\,\sum_{s:\,\Delta S_s(\omega)\neq 0}\delta_{(\Delta S_s(\omega),s)}(\D y,\D s)$$ of the process
	$\process{S}$, and let $\process{C}=\{\bigl(C_t^{ij}\bigr)_{1\leq i,j\leq d})\}_{t\geq0}$ be the quadratic co-variation
	process for $\process{S^{c}}$ (continuous martingale part of
	$\process{S}$), that is,
	$C^{ij}_t=\langle S^{i,c}_t,S^{j,c}_t\rangle.$  Then $(B,C,N)$ is called
	the \emph{characteristics} of the semimartingale $\process{S}$
	(relative to $h(x)$). In addition, by defining $\tilde{C}(h)^{ij}_t:=\langle
	M(h)^{i}_t,M(h)^{j}_t\rangle$, $i,j=1,\ldots,d$, where $\process{M(h)}$ is the local martingale
	appearing in \cref{SS},  $(B,\tilde{C},N)$ is called the \emph{modified
		characteristics} of the semimartingale $\process{S}$ (relative to $h(x)$).
\end{definition}

Now, according to  \cite[Theorem 3.5]{rene-holder} and \cite[Proposition II.2.17]{jacod} we see that the (modified) characteristics of a LTP $\{X_t\}_{t\ge0}$ (with respect to a truncation function $h(x)$)  are given by
\begin{align*}B(h)^{i}_t&\,=\,\int_0^{t}b_i(X_{s})\,\D s+\int_0^{t}\int_{\R^d}\left(h_i(y)-y_i\Ind_{B_1(0)}(y)\right)\nu(X_{s},\D y)\,\D s\,,\\
C^{ij}_t&\,=\,\int_0^{t}c_{ij}(X_{s})\,\D s\,,\\
N(\D y,\D s)&\,=\,\nu(X_{s},\D y)\,\D s\,,\\
\tilde{C}(h)^{ij}_t&=\int_0^{t}c_{ij}(X_{s})\,\D s+\int_0^{t}\int_{\R^{d}}h_{i}\left( y\right)h_{j}\left( y\right)\nu(X_{s},\D y)\,\D s\,,
\end{align*} for $t\geq0$ and $i,j=1,\ldots,d.$

In the sequel, we  assume that $\process{X}$ admits ``small jumps'' only, that is,

\smallskip

\begin{description}
	\item [(\textbf{C3})]  $\displaystyle \sup_{x\in \R^d}\int_{\R^{d}}|y|^2\nu(x,\D y)<\infty$.
\end{description}
\smallskip

\noindent
As a direct consequence of (\textbf{C3}) and \cite[Proposition II.2.29]{jacod} we see that  $\process{X}$ itself is a special semimartingale, and for the truncation function we can take $h(x)=x$. In particular, if $\nu(x,\D y)$ is also symmetric for every $x\in\R^d$, the first characteristic
$B(h)^{i}_t$
equals to $\int_0^{t}b_i(X_{s})\,\D s$ for $t\ge0$ and $i=1,\dots,d$.

 Observe next that    $\process{X}$  is  a Hunt process (since it is Feller). Thus, $\process{X}$ is an It\^{o} process in the sense of \cite{cinlar} (a semimartngale Hunt process with characteristics of the form as above).
Now,   \cite[Theorem 3.33]{cinlar} asserts that
there exist  a suitable enlargement of the stochastic basis $(\Omega,\mathcal{F},\{\mathbb{P}_{x}\}_{x\in\R^{d}},\process{\mathcal{F}},\process{\theta})$, say $(\widetilde{\Omega},\mathcal{\widetilde{F}},\{\mathbb{\widetilde{P}}_{x}\}_{x\in\R^{d}},\process{\mathcal{\widetilde{F}}},\process{\widetilde{\theta}})$, supporting a $d$-dimensional Brownian motion $\process{\tilde{W}}$ and a Poisson random measure $\tilde{\mu}(\cdot,\D z,\D s)$ on
 $\mathcal{B}(\R)\otimes\mathcal{B}([0,\infty))$ with compensator
 $\tilde{\nu}(\D z)\,\D s$, such that $\{X_t\}_{t\ge0}$ is
a solution to the following stochastic differential equation
\begin{align*}X_t\,=\,&x+\int_0^{t}b(X_{s
})\,\D s+\int_0^{t}\tilde\sigma(X_{
s})\,\D \tilde{W}_s\\&+\int_0^{t}\int_{\R}k(X_{s-},z)\Ind_{\{u:|k(X_{s-},u)|< 1\}}(z)\left(\tilde{\mu}(\cdot,\D z,\D s)-\tilde{\nu}(\D z)\,\D s\right)\\&+\int_0^{t}\int_{\R}k(X_{s-},z)\Ind_{\{u:|k(X_{s-},u)|\ge1\}}(z)\,\tilde{\mu}(\cdot,\D z,\D s)\,,
\end{align*}
where
$\tilde\sigma(x)$
is a $d \times d$ matrix-valued Borel measurable function such that $\tilde\sigma(x)'\tilde\sigma(x)=c(x)$ for any $x\in \R^d$,
$\tilde{\nu}(\D z)$ is any given $\sigma$-finite non-finite and non-atomic measure  on $\mathcal{B}(\R)$, and
$k:\R^{d}\times\R\to\R^{d}$ is a Borel measurable function  satisfying
$$\mu(\cdot,\D y,\D s)\,=\, \tilde{\mu}\bigl(\cdot,\{(z,u)\in\R\times[0,\infty):(k(X_{u-},z),u)\in(\D y,\D s)\}\bigr)\,,$$
and
$$\nu\bigl(x,\D y\bigr)\,=\,\tilde{\nu}\bigl(\{z\in\R:k(x,z)\in \D y\}\bigr)\,.$$
Thus, due to this and (\textbf{C3})
we have that
\begin{equation}\begin{aligned}\label{SDE}
X_t\,=\,&x+\int_0^{t}b(X_{s
})\D s+\int_0^{t}\tilde\sigma(X_{s
})\D \tilde{W}_s\\&+\int_0^{t}\int_{\R}k(X_{s
},z)\Ind_{\{u:|k(X_{s
},u)|\ge1\}}(z)\,\tilde{\nu}(\D z)\,\D s\\&+\int_0^{t}\int_{\R}k(X_{s-},z)\left(\tilde{\mu}(\cdot,\D z,\D s)-\tilde{\nu}(\D z)\, \D s\right)\,.
\end{aligned}\end{equation}
From this equation we also read the unique special semimartingale  decomposition  of $\{X_t\}_{t\ge0}$.

\subsection*{Main
Result}
Before stating the main result of this article, we  introduce some notation we need.
 Denote by   $C_b^k(\R^d)$ with  $k\in\N_0
 :=\{0,1,2,\dots\}$
  the space of $k$ times differentiable functions such that all derivatives up to order $k$ are bounded. This space is a Banach space endowed with the norm $\lVert f\rVert_k:=\sum_{m:\,|m|\le k}\lVert D^m f\rVert_\infty$, where $m=(m_1,\dots,m_d)\in\N^d_0$, $|m|:=m_1+\cdots+m_d,$ and $D^m:=\partial^{m_1}\dots\partial^{m_d}.$ Denote also $C_b^\infty(\R^d):=\cap_{k\in\N_0}C_b^k(\R^d)$.
 Further, a function $\phi:(0,1]\to(0,\infty)$ is said to be \textit{almost increasing} if there is
 a constant
 $\underline{\kappa}\in(0,1]$ such that $\underline{\kappa}\,\phi(r)\le\phi(R)$ for all $r,R\in(0,1]$
  with $r\le R$. Analogously, $\phi:(0,1]\to(0,\infty)$ is said to be \textit{almost decreasing} if there is
 a constant
 $\overline{\kappa}\in[1,\infty)$ such that $\phi(R)\le\overline{\kappa}\,\phi(r)$ for all $r,R\in(0,1]$
 with $r\le R$.
Let now $\psi:(0,1]\to[0,\infty)$ be such that $\psi(1)=1$ and $\lim_{r\to0}\psi(r)=0$.  For $f\in C_b(\R^d)$ and $j\in\N_0$,  define $$[f]_{-j,\psi}\,:=\,\sup_{x\in\R^d}\sup_{h\in \bar {B}_1(0)\setminus\{0\}}\frac{|f(x+h)-f(x)|}{\psi(|h|)|h|^{-j}}\,,$$ where $\bar {B}_r(x)$ stands for the (topologically) closed ball of radius $r$ around $x\in\R^d.$ Also, let \begin{align*}
m_\psi\,:=\,\sup\{\alpha\in\R:\,r\mapsto\psi(r)/r^\alpha\ \text{is almost increasing in}\ (0,1]\}\,,\\
M_\psi\,:=\,\inf\{\alpha\in\R:\,r\mapsto\psi(r)/r^\alpha\ \text{is almost decreasing in}\ (0,1]\}\,.
\end{align*} According to \cite[Theorem 2.2.2]{goldie}, $m_\psi\le M_\psi$. If $m_\psi>0$,  we call $\psi(r)$ the \textit{H\"{o}lder exponent}. In this case, if $m_\psi\in(k,k+1]$ for some $k\in\N_0$, define $$C_b^\psi(\R^d)\,:=\,\{f\in C^k_b(\R^d):\, [D^mf]_{-k,\psi}<\infty\ \text{for}\ |m|=k \}\,.$$ This space is called  a \textit{generalized H\"{o}lder space}, and it is a normed vector space with the norm  $$\lVert f\rVert_\psi\,:=\,\lVert f\rVert_{k}+\sum_{m:\,|m|=k}[D^mf]_{-k,\psi}\,,$$
(see \cite{kassmann2}).
Observe that the product of two H\"{o}lder exponents is a H\"{o}lder exponent, and
that if $m_\psi\in(k,k+1]$ for some $k\in\N_0$ then $C^{k+1}_b(\R^d)\subsetneq C_b^\psi(\R^d)\subsetneq C_b^k(\R^d)$.
  In particular, when $\psi(r)=r^\gamma$ for some $\gamma>0$,  $C_b^\psi(\R^d)$ becomes the classical H\"{o}lder space of order $\gamma$ (usually denoted by  $C_b^\gamma(\R^d)$), which is a Banach space together with the above-defined norm (which we denote by $\lVert\cdot\rVert_\gamma$).
Since $f\leftrightarrow f_\tau$ gives a one-to-one correspondence between $\{f:\R^d\to\R:f\ \text{is}\ \tau\text{-periodic}\}$ and $\{f_\tau:\mathbb{T}_\tau^d\to\R\}$, in  an analogous way we define $C^k(\mathbb{T}^d_\tau)$ and $C^\psi(\mathbb{T}^d_\tau)$.

We are now in position to state the main result of this article, the proof of which  is given in
\Cref{S2}.
\begin{theorem}\label{T1.1}Let  $\process{X}$ be a $d$-dimensional LTP with  semigroup $\process{P}$, symbol $q(x,\xi)$ and L\'evy triplet $(b(x),c(x),\nu(x,\D y))$,
	satisfying $($\textbf{C1}$)$, $($\textbf{C2}$)$, $($\textbf{C3}$)$ and
	
	\medskip
	
	\begin{description}
		\item [(\textbf{C4})]
		 $\displaystyle x\mapsto b^*(x):=b(x)+\int_{B_1^c(0)}y\,\nu(x,\D y)$ is of class $C_b^\psi(\R^d)$ for some H\"{o}lder exponent $\psi(r)$, and
		
		 \medskip
		
			\begin{itemize}
				\item[(i)] for some $t_0>0$, any $t\in(0, t_0]$ and any $\tau$-periodic $f\in C_b(\R^d)$,
				$$\|P_t f\|_{\psi}\,\le\, C(t) \|f\|_\infty\,,$$ where $\int_0^{t_0} C(t)\,\D t <\infty;$
				
				\medskip
				
				\item[(ii)] for some $\lambda>0$ and any $\tau$-periodic $f\in C_b^\psi(\R^d)$ with $\int_{\mathbb{T}_\tau^d}f_\tau(x)\,\pi(\D x)=0$, the Poisson equation
				\begin{equation}\label{e:po-2} \lambda u- \mathcal{A}^b u\,=\,f  \end{equation} admits a  $\tau$-periodic solution $u_{\lambda,f}\in C_b^{\varphi\psi}(\R^d)$ for some H\"{o}lder exponent $\varphi(r)$.
\end{itemize}
			\end{description}

	\smallskip

\noindent	Then,

\medskip

\begin{itemize}
	\item [(a)] the Poisson equation \begin{equation}\label{e:po-1}
\mathcal{A}^b\beta=b^*-\bar {b^*}
	\end{equation} admits a $\tau$-periodic solution $\beta\in C_b^{\varphi\psi}(\R^d)$. Moreover, $\beta(x)$ is the unique solution in the class of continuous and periodic solutions  to \cref{e:po-1} satisfying
$\int_{\mathbb{T}_\tau^d}\beta_\tau(x)\,\pi(\D x)=0$.
	
	\medskip

	\item[(b)] in any of the following three cases
	
	\medskip
	
	\begin{itemize}
		\item [(1)] $\beta\in C_b^2(\R^d)$
		if $c(x)\not\equiv0$;
		
		\medskip
	
		\item[(2)] $m_{\varphi\psi}>1$  if $c(x)\equiv0$ and \begin{equation}\label{eq:it}
		\sup_{x\in \R^d}\int_{B_1(0)}\varphi(|y|)\psi(|y|)\,\nu(x,\D y)\,<\,\infty\,;
		\end{equation}
		
		\medskip \item [(3)] $\beta\in C^1_b(\R^d)$ if $c(x)\equiv0$ and \begin{equation}\label{eq:it2}
		\sup_{x\in \R^d}\int_{B_1(0)}|y|\,\nu(x,\D y)\,<\,\infty\,,
		\end{equation}
	\end{itemize}
		for any initial
	distribution  of $\process{X}$,
	\begin{equation}\label{ET1.1}\left\{\varepsilon X_{\varepsilon^{-2}t}-\varepsilon^{-1}\bar {b^*} t\right\}_{t\geq0}\xRightarrow[]{\varepsilon\to0}\process{W}\,.\end{equation}
	\end{itemize}

 \medskip

\noindent	Here, $$ \bar {b^*}\,:=\,\int_{\mathbb{T}^d_\tau}b^*(x)\,\pi(\D x)\,,$$ $\Rightarrow$ denotes the convergence in the space of c\`adl\`ag functions
	endowed with the Skorohod  ${\rm J}_1$-topology, and
	$\process{W}$ is a  $d$-dimensional zero-drift Brownian motion
	determined by  covariance matrix  $\Sigma$ given in \cref{ET1.2}.
 \end{theorem}

 Under (\textbf{C1}), (\textbf{C2}) and the assumption that $b^*\in C_b(\R^d)$, in \Cref{lm:cont} below we show that \cref{e:po-1}
   admits a $\tau$-periodic solution $\beta\in C_b(\R^d)$ (which is also unique in the class of continuous $\tau$-periodic solutions satisfying
   $\int_{\mathbb{T}_\tau^d}\beta_\tau(x)\,\pi(\D x)=0$). However, we require
  additional smoothness of $\beta(x)$ in  order to apply It\^{o}'s formula (given in  \Cref{Ito}) in the proof of \Cref{T1.1} (see \Cref{S2} for details).
  This additional regularity is given through
 (\textbf{C4}) (together with (\textbf{C1}) and (\textbf{C2})). Namely, under these assumptions,  we show that  $\beta\in C_b^{\varphi\psi}(\R^d)$.
 When $c(x)\not\equiv0$ we require $\beta\in C_b^2(\R^d)$, and when $c(x)\equiv0$ and $b(x)\not\equiv0$  or $\nu(x,\D y)$ is non-symmetric for some $x\in\R^d$ we only require $m_{\varphi\psi}>1$
or $\beta\in C_b^1(\R^d)$.
 When $b(x)\equiv0$ and $\nu(x,\D y)$ is symmetric for all $x\in\R^d$,  as already commented,  $\beta(x)\equiv0$ and the assertion of the theorem follows without assuming (\textbf{C4}).
 In the  pure-jump case (that is, when $c(x)\equiv0$), \cref{eq:it} suggests that
   the H\"{o}lder exponent
   $\varphi(r)$ depends on the behavior  of $\nu(x,\D y)$ on $B_1(0)$.
   For example,
   when \begin{equation}\label{ER1.3}\frac{\underline\kappa}{|y|^{d}
   \varphi(|y|)}\Ind_{B_1(0)}(y)\,\D y\,\le\,\Ind_{B_1(0)}(y)\,\nu(x,\D y)\,\le\, \frac{\overline\kappa}{|y|^{d}
   \varphi(|y|)}\Ind_{B_1(0)}(y)\,\D y\,,\end{equation}
     for  some
   $0<\underline\kappa\le\overline{\kappa}<\infty$, \cref{eq:it} trivially holds true. Thus, we only require that $\beta\in C_b^{
   \varphi\psi}(\R^d)$ for some  H\"{o}lder exponent $\psi(r)$ with $m_{\varphi\psi}>1$.
Analogously, if \cref{eq:it2}  holds true,
then we only require that $\beta\in C_b^{
		1}(\R^d)$.
Observe that in the pure-jump case we do not require explicitly that
$\beta\in C_b^2(\R^d)$.
 In this sense the  assumption $u_{\lambda,f}\in C_b^{\varphi\psi}(\R^d)$ in ({\bf C4})(ii) is optimal. Namely, under \cref{eq:it}
(resp.\ \cref{eq:it2}), in \Cref{Ito} we show that when $m_{\varphi\psi}>1$ (resp.\ $\beta\in C_b^1(\R^d))$ we can apply It\^{o}'s formula to the process $\{\beta(X_t)\}_{t\ge0}$.

	Let us also remark that
		in the  proof of \Cref{T1.1} (a) we show that ({\bf C4})(i) (together with ({\bf C1})-({\bf C3})) implies that $\beta\in C_b^{\psi}(\R^d).$  Hence, \cref{ET1.1} holds true if $\psi(r)$ is such that either (1), (2) (with $\varphi(r)\equiv 1$) or (3) above is satisfied.  If this is not the case, then we require an additional regularity of $\beta(x)$ (inherited from the semigroup) which is given through ({\bf C4})(ii).

 Several examples of LTPs satisfying ({\bf C4}) (and ({\bf C1})-({\bf C3})) are presented in \Cref{examples}. In particular, if $\process{X}$ is a diffusion process with  $\tau$-periodic coefficients $b\in C_b^\varepsilon(\R^d)$ and $c\in C_b^{1+\varepsilon}(\R^d)$ for some $\varepsilon\in(0,1)$, and additionally $c(x)$ being positive definite,  (\textbf{C4}) with $\psi(r)=r^\varepsilon$ and $\varphi(r)=r^2$   follows from  \cite[Theorem 2.1]{tomisaki}.
In particular,   \Cref{T1.1} generalizes the results  from
		\cite{benso-lions-book,bhat2} where periodic homogenization of a diffusion process with $\tau$-periodic coefficients $b\in C_b^1(\R^d)$ and  $c\in C_b^{2}(\R^d)$, and $c(x)$ being positive definite, has been considered.

 If $\process{X}$ is a LTP with diffusion and drift coefficients as above, and L\'evy kernel satisfying \cref{ER1.3} (and some additional regularity properties discussed in \Cref{examples}), then (\textbf{C4}) with $\psi(r)=r^\varepsilon$  and $\varphi(r)=r^2$ follows again from \cite[Theorem 2.1]{tomisaki}.

 If $\process{X}$ is a pure-jump LTP with vanishing drift term and L\'evy kernel satisfying \cref{ER1.3} (and some additional regularity properties discussed in \Cref{examples}),   (\textbf{C4}) holds true for any H\"{o}lder exponent $\psi(r)$ such that $[m_\psi,M_\psi]\subset (0,1)$ and $[m_{\varphi\psi},M_{\varphi\psi}]\cap \N=\emptyset$  (see \Cref{examples}).

\subsection*{Literature Review}

Our work relates to the active research on homogenization of integro-differential operators, and Markov processes with jumps. The work  is highly motivated
by  the results in \cite{benso-lions-book,bhat2,hom} where, by employing   probabilistic techniques, the authors considered  periodic homogenization   of the operator $\mathcal{L}$ with $\nu(x,\D y)\equiv0$ (that is, second-order elliptic operator in non-divergence form), and   $\mathcal{L}$ with $b(x)\equiv0$ and $\nu(x,\D y)$ being symmetric for all $x\in\R^d$ (that is, integro-differential operator in the balanced form), respectively. In
this article, we generalize both results by including the non-local part of the operator $\mathcal{L}$, as well as non-symmetries caused by the drift term $b(x)$ and the L\'evy kernel $\nu(x,\D y)$.
In a closely related work
\cite{zhizhina}, by using analytic techniques (the corrector method), the authors discuss periodic homogenization  of the operator $\mathcal{L}$ with a convolution-type L\'evy kernel, that is, $\mathcal{L}$ is determined by
$$\nu(x,\D y)\,=\,\lambda(x)\mu(x+y)a(y)\,\D y\qquad \text{and}\qquad b(x)\,=\,\int_{B_1(0)}y\,\nu(x,\D y)$$ with
$\lambda(x)$ and $\mu(x)$ being measurable, $\tau$-periodic and such that $0<\underline\kappa\le\lambda(x),\mu(x)\le \overline{\kappa}<\infty$ for all $x\in\R^d$, and $a(y)\ge0$ being measurable and such that $0<\int_{\R^d}(1\vee|y|^2)a(y)\,\D y<\infty$ and $a(y)=a(-y)$ for all $y\in\R^d$. The homogenized operator is again a second-order elliptic operator with constant coefficients.  Observe that this case is not covered by \Cref{T1.1} since finiteness of  $\nu(x,\D y)$ excludes regularity properties of the corresponding semigroup assumed in ({\bf C4}).

There is a vast literature
on homogenization of differential operators, mostly based on PDE methods.
We refer the interested readers to  \cite{benso-lions-book,chechkin,donato,kozlov,tartarII} and the references therein.
Results related to the problem of periodic homogenization of non-local operators (based on probabilistic techniques) were obtained in \cite{franke-periodic,franke-periodicerata, franke2,fujiwara,horie-inuzuka-tanaka,huang-duan-song,huang-duan-song2,Tom}. In all this works the focus is on the so-called stable-like operators
(possibly with variable order),
 that is, on the case when  $\nu(x,\D y)$ admits ``large jumps'' of power-type: $$\frac{\underline\kappa}{|y|^{d+\overline \alpha}}\Ind_{B_1^c(0)}(y)\,\D y\,\le\,\Ind_{B_1^c(0)}(y)\,\nu(x,\D y)\,\le\,\frac{\overline{\kappa}}{|y|^{d+\underline\alpha}}\Ind_{B_1^c(0)}(y)\,\D y,\qquad x\in \R^d\,,$$ for some $0<\underline\kappa\le\overline{\kappa}<\infty$ and $0<\underline\alpha\le\overline\alpha<2$. In this case, by using subdiffusive scaling (which depends on the behavior of $\nu(x,\D y)$ on $B_1^c(0)$), the homogenized operator is the infinitesimal generator of a stable L\'evy process with the index of stability being equal to the power of the scaling factor. The problem of stochastic homogenization (that is, homogenization of operators with random coefficients) of this type of operators has been considered in \cite{rhodes}.
PDE and
other
 analytical approaches to the problem of
periodic homogenization and
stochastic homogenization of stable-like operators  can be found in \cite{ari2,ari,ari3,imbert,gosh,salort,focardi, kassmann,uemura,schwabI,schwabII}.

  Let us also remark that the class of processes considered in the present article constitute of both diffusion and pure-jump part, and the behavior of the homogenized process depends on both of them. This makes the approach to this problem more subtle since we need to take care of  diffusion processes, diffusion processes with jumps and pure jump processes, simultaneously.

\section{Proof of \Cref{T1.1}
}\label{S2}

Throughout
this section we assume that $\process{X}$ is a $d$-dimensional LTP with semigroup $\process{P}$, symbol $q(x,\xi)$ and L\'evy triplet $(b(x),c(x),\nu(x,\D y))$, satisfying (\textbf{C1})-(\textbf{C4}). A crucial step in the proof is an application of It\^{o}'s formula. In order to justify this step, we first discuss  regularity of a solution to the Poisson equation  \cref{e:po-1}.

\subsection*{Solution to the Poisson Equation \cref{e:po-1}}
 Observe first  that for any  $f_\tau\in B_b(\mathbb{T}_\tau^d)$ with $\int_{\mathbb{T}_\tau^d} f_\tau(x)\,\pi(\D x)=0,$  \Cref{SS2.1} implies that
 $$\|P^\tau_t f_\tau\|_\infty \,\le\ \Gamma \E^{-\gamma t} \|f_\tau\|_\infty\,,\qquad t\ge0\,.$$ In particular,
 $$\left\|\int_0^\infty P^\tau_t f_\tau \,\D t\right\|_\infty\,
 \le \,\frac{\Gamma}{\gamma}\lVert f_\tau\rVert_\infty \,<\,\infty\,.$$
 Therefore, the zero-resolvent
 $$R^\tau f_\tau(x)\,:=\,\int_0^\infty  P^\tau_t f_\tau(x) \,\D t\,,\qquad x\in\mathbb{T}_\tau^d\,,$$ is well defined, and
 $$\int_{\mathbb{T}_\tau^d}R^\tau f_\tau(x)\,\pi(\D x)\,=\,0\,.$$

According  to \cite[Corollary 3.4]{rene-conserv}, $\process{X}$ is a $C_b$-Feller process. Thus, $\{\Pi_\tau(X_t)\}_{t\ge0}$ is also $C_b$-Feller, and  $R^\tau f_\tau\in C(\mathbb{T}_\tau^d)$ for every $f_\tau\in C(\mathbb{T}_\tau^d)$ satisfying $\int_{\mathbb{T}^d_\tau}f_\tau(x)\,\pi(\D x)=0$.  Since $\mathbb{T}_\tau^d$ is compact, $\{\Pi_\tau(X_t)\}_{t\ge0}$ is a Feller process. Denote the corresponding Feller generator by $(\mathcal{A}_\tau^\infty,\mathcal{D}_{\mathcal{A}_\tau^\infty})$. Clearly,
 for any $f_\tau\in\mathcal{D}_{\mathcal{A}_\tau^\infty}$(which is by definition  continuous), and its $\tau$-periodic extension $f(x)$, it holds that $ f\in\mathcal{D}_{\mathcal{A}^b}$ and $\mathcal{A}_\tau^\infty f_\tau=\mathcal{A}^b f.$
  It is clear now that $R^\tau f_\tau\in \mathcal{D}_{\mathcal{A}^\infty_\tau}$ for any $f_\tau\in C(\mathbb{T}_\tau^d)$ with $\int_{\mathbb{T}_\tau^d} f_\tau(x)\,\pi(\D x)=0.$

 Now, we turn to the  Poisson equation
\cref{e:po-1}.
 Denote by $b^*_\tau(x)$ the restriction of $b^*(x)$ to $\mathbb{T}_\tau^d$, and set $\bar {b^*_\tau}=\int_{\mathbb{T}_\tau^d}b^*_\tau(x)\,\pi(\D x)$. By
 assumption $ b^*_\tau\in C^{\psi}(\mathbb{T}_\tau^d)$.
 Define now $\beta_\tau(x):=-R^\tau (b^*_\tau(\cdot)-\bar b^*_\tau)(x)$ for any $x\in\mathbb{T}_\tau^d$.
According to the argument above, we immediately get the following.
 \begin{lemma}\label{lm:cont} The $\tau$-periodic extension $\beta(x)$ of $\beta_\tau(x)$ is continuous  and satisfies  \cref{e:po-1}. Moreover, $\beta(x)$ is the  unique solution in the class of continuous and $\tau$-periodic solutions to \cref{e:po-1} satisfying
 $\int_{\mathbb{T}_\tau^d} \beta_\tau(x)\,\pi(\D x)=0$.\end{lemma}
\begin{proof}We only need to prove uniqueness.
	Let $\bar \beta(x)$ be another continuous and $\tau$-periodic solution to \cref{e:po-1} satisfying
 $\int_{\mathbb{T}_\tau^d} \bar\beta_\tau(x)\,\pi(\D x)=0$. Then, $\mathcal{A}^b(\beta-\bar{\beta})(x)\equiv0$. In particular, according to \cite[Proposition 4.1.7]{ethier},
	$$(\beta-\bar{\beta})(x)\,=\,\mathbb{E}_x\bigl[(\beta-\bar{\beta})(X_t)\bigr]\,=\,\mathbb{E}^\tau_{x_\tau}\bigl[(\beta_\tau-\bar{\beta}_\tau)(\Pi_{\tau}(X_t))\bigr]\,,\qquad x\in\R^d\,,\ t\ge0\,.$$ By letting now $t\to\infty$, it follows from \Cref{SS2.1} that $(\beta-\bar{\beta})(x)\equiv0$, which proves the assertion.
	\end{proof}

Observe that in \Cref{lm:cont} we only used the fact that $b^*_\tau\in C(\mathbb{T}_\tau^d)$.
In the sequel, we discuss additional smoothness of $\beta(x)$.

\subsection*{Proof of  \Cref{T1.1} (a)}
We first claim that for any $\tau$-periodic $f\in C_b(\R^d)$ such that $\int_{\mathbb{T}_\tau^d}f_\tau(x)\,\pi(\D x)=0$, $R^\tau f_\tau\in C^\psi(\mathbb{T}_\tau^d)$. Indeed, by  ({\bf C4})(i), we have
$$\int_0^{t_0} \|P_t^\tau f_\tau\|_{\psi}\,\D t  \,\le\ \|f_\tau\|_\infty\int_0^{t_0} C(t)\,\D t \,<\,\infty\,.$$ Also,
since for any $t>0$ and any $f_\tau\in C(\mathbb{T}_\tau^d)$ with $\int_{\mathbb{T}_\tau^d}f_\tau(x)\,\pi(\D x)=0$, $P^\tau_tf_\tau\in C(\mathbb{T}_\tau^d)$ and $\int_{\mathbb{T}_\tau^d}P_t^\tau f_\tau(x)\,\pi(\D x)=0$,  ({\bf C4})(i) and \Cref{SS2.1} imply that
$$\int_{t_0}^\infty \|P_t^\tau f_\tau\|_{\psi}\,\D t \,\le\, C(t_0)\int_{t_0}^\infty \|P^\tau_{t-t_0}f_\tau\|_\infty\,\D t\,\le\, \Gamma C(t_0)\|f_\tau\|_\infty\int_{t_0}^\infty \E^{-\lambda (t-t_0)}\,\D t \,<\,\infty\,.$$ Combining both estimates above with the fact that $R^\tau f_\tau=\int_0^\infty P_t^\tau f_\tau\,\D t$, we get
$R^\tau f_\tau\in C^\psi(\mathbb{T}_\tau^d)$.

Finally, for $\lambda> 0$ let $R_\lambda^\tau $ be the $\lambda$-resolvent of $\{\Pi_\tau(X_t)\}_{t\ge0}$.
  Clearly, for  any $\tau$-periodic $f\in C_b^\psi(\R^d)$ satisfying $\int_{\mathbb{T}_\tau^d}f_\tau(x)\,\pi(\D x)=0$,
 the $\tau$-periodic extension of $R_\lambda^\tau  f_\tau(x)$ (say $\bar u_{\lambda,f}(x)$) is a solution to   \cref{e:po-2}. Observe  next that necessarily $u_{\lambda,f}(x)=\bar u_{\lambda,f}(x)$ for all $x\in\R^d$. Namely, since $\bar u_{\lambda,f}\in C_b(\R^d)$,  and by \eqref{e:po-2},
 \begin{align*} u_{\lambda,f}(x)-\bar u_{\lambda,f}(x)&\,=\, \E^{-\lambda t}\,\mathbb{E}_x\bigl[(u_{\lambda,f}-\bar u_{\lambda,f})(X_t)\bigr]\\&\ \ \ \ \ -\int_0^t \E^{-\lambda s}\,\mathbb{E}_x\bigl[\mathcal{A}^b(u_{\lambda,f}-\bar u_{\lambda,f})(X_s)-\lambda(u_{\lambda,f}-\bar u_{\lambda,f})(X_s)\bigr]\,\D s\\
&\,=\, \E^{-\lambda t}\,\mathbb{E}_x\bigl[(u_{\lambda,f}-\bar u_{\lambda,f})(X_t)\bigr]\,,\qquad x\in\R^d\,,\ t\ge0\,,\end{align*}
 by letting $t\to\infty$ the assertion follows.
Thus, since $b^*_\tau\in C_b^\psi(\mathbb{T}_\tau^d)$,
from the  resolvent identity
\begin{equation*}\label{e:res}R^\tau (b^*_\tau(\cdot)-\bar{b}^*_\tau)(x)\,=\,R_\lambda^\tau((b^*_\tau(\cdot)-\bar{b}^*_\tau)+\lambda R^\tau (b^*_\tau(\cdot)-\bar{b}^*_\tau))(x)\end{equation*}
 and ({\bf C4})(ii), we conclude the result. \hfill\(\Box\)

\bigskip

In \Cref{T1.1} (b)  we require that $\beta\in C_b^2(\R^d)$
if $c(x)\not\equiv 0$, and that $m_{\varphi\psi}>1$  (resp.  $\beta\in C_b^1(\R^d)$) if $c(x)\equiv0$  and \cref{eq:it} (resp. \cref{eq:it2}) holds true.
In the following proposition we slightly generalize \cite[Lemma 4.2]{priola} (see also \cite{errami}), and prove It\^{o}'s formula for a pure-jump LTP with respect to a not necessarily twice continuously differentiable function.

\begin{proposition}\label{Ito}
	Assume that $\process{X}$ is pure-jump
$($that is, $c(x)\equiv0$$)$ and that
there is a H\"{o}lder exponent $\phi(r)$ with $m_\phi>1$ such that
\begin{equation}
	\label{delta}
	\sup_{x\in\R^d}\int_{B_1(0)}
\phi(|y|)\,\nu(x,\D y)\,<\,\infty.\end{equation}
Then, it holds that \begin{align*}f(X_t)
	\,=\,& f(X_0)+\int_0^t \bigl\langle\nabla f(X_{s
}),b^*(X_{s
})\bigr\rangle\,\D s\\&+ \int_0^t\int_{\R^d}\left(f(X_{s-}+k(X_{s-},z))-f(X_{s-})\right)\bigl(\tilde
\mu(\cdot, \D z,\D s)- \tilde\nu(\D z)\, \D s\bigr)
	\\
	& + \int_0^t\int_{\R^d}\Big(f(X_{s}+k(X_s,z))-f(X_{s})  -\langle\nabla f(X_{s}), k(X_s,z)\rangle\Big)
\tilde\nu(\D z)\,\D s
\,,\end{align*}  for all $t\ge0$ and all $f\in C_b^{\phi}(\R^d)$,
where $b^*(x)=b(x)+\int_{B_1^c(0)}y\,\nu(x,\D y)$.
In addition, if \cref{eq:it2} holds true, then the above relation holds for any $f\in C_b^1(\R^d)$.
	\end{proposition}
\begin{proof} Without loss of generality, assume that $m_\phi\in(1,2]$.
	We follow the proof of \cite[Lemma 4.2]{priola}. Let $f\in C_b^{\phi}(\R^d)$, and let $\chi\in C_c^\infty(\R^d)$, $0\le\chi\le1$, be such that $\int_{\R^d}\chi(x)\,\D x=1$. For $n\in\N$ define $\chi_n(x):=n^d\chi(nx)$, and $f_n(x):=(\chi_n\ast f)(x)$, where $\ast$ stands for the standard convolution operator. Clearly, $\{f_n\}_{n\in\N}\subset
C_b^\infty(\R^d)$, $\lVert f_n\rVert_{\phi}\le\lVert f\rVert_{\phi}$ for $n\in\N$,
$\lim_{n\to\infty}f_n(x)=f(x)$
and $\lim_{n\to\infty}\nabla f_n(x)=\nabla f(x)$ for all $x\in\R^d$.  Next, by employing It\^{o}'s formula and \cref{SDE} we have that
\begin{align}\label{e:sde}
	f_n(X_t)
	\,=\,& f_n(X_0)+\int_0^t \bigl\langle\nabla f_n(X_{s
}),b
(X_{s
})\bigr\rangle\,\D s\nonumber\\&
+\int_0^{t}\int_{\R}\langle \nabla f_n(X_{s}), k(X_{s
},z)\rangle \Ind_{\{u:|k(X_{s
},u)|\ge1\}}(z)\,\tilde{\nu}(\D z)\,\D s\nonumber\\
&+ \int_0^t\int_{\R}\left(f_n(X_{s-}+k(X_{s-},z))-f_n(X_{s-})\right)\bigl(\tilde\mu(\cdot,\D z,\D s)-\tilde \nu(\D z)\,\D s\bigr)\nonumber
	\\
	& +\int_0^t\int_{\R}\Big(f_n(X_{s}+k(X_{s},z))-f_n(X_{s})  -\langle\nabla f_n(X_{s}), k(X_{s},z)\rangle\Big)\,\tilde \nu(\D z)\,\D s\nonumber\\
\,=\,&f_n(X_0)+\int_0^t \bigl\langle\nabla f_n(X_{s
}),b^*
(X_{s
})\bigr\rangle\,\D s\nonumber\\
&+ \int_0^t\int_{\R}\left(f_n(X_{s-}+k(X_{s-},z))-f_n(X_{s-})\right)\bigl(\tilde\mu(\cdot,\D z,\D s)-\tilde \nu(\D z)\,\D s\bigr)\\
&+\int_0^t\int_{\R^d}\Big(f_n(X_{s}+k(X_s,z))-f_n(X_{s})  -\langle\nabla f_n(X_{s}), k(X_s,z)\rangle\Big)\, \tilde\nu(\D z)\,\D s.\nonumber
	\end{align}
Now, by letting $n\to\infty$ we see that the left hand-side converges to $f(X_t)$, and the first two terms on the right-hand side (for the second term we employ the dominated convergence theorem) converge to $f(X_0)$ and $\int_0^t \bigl\langle\nabla f
(X_{s-}),b^*(X_{s-})\bigr\rangle\,\D s$, respectively. Further,  Taylor's theorem
together with the fact that $m_\phi>1$ implies
 \begin{align*}&|f_n(x+y)-f_n(x)  -\langle\nabla f_n(x), y\rangle|\\&\,\le\,\int_0^1|\nabla f_n(x+ry)-\nabla f_n(x)||y|\,\D r\\&\,\le\,\lVert f_n\rVert_{\phi}\phi(|y|)\Ind_{B_1(0)}(y)+2\lVert f_n\rVert_{\phi}|y|\Ind_{B^c_1(0)}(y)
\\&\,\le\,\lVert f\rVert_{\phi}\phi(|y|)\Ind_{B_1(0)}(y)+2\lVert f\rVert_{\phi}|y|\Ind_{B^c_1(0)}(y)\,.\end{align*} Observe that according to \cref{delta} and ({\bf C3}), $$\sup_{x\in \R^d}\int_{\R^d}\left(\phi(|y|)\Ind_{B_1(0)}(y)+|y|\Ind_{B_1^c(0)}(y)\right)\,\nu(x,\D y)<\infty\,.$$ Thus, the dominated convergence theorem implies that the last term converges to $$\int_0^t\int_{\R^d}\Big(f(X_{s}+k(X_s,z))-f(X_{s})  -\langle\nabla f(X_{s}), k(X_s,z)\rangle\Big)
\tilde\nu(\D z)\,\D s\,.$$ Finally, by employing the isometry formula,
we have
\begin{align*}
&
\,\widetilde{\mathbb{E}}_x\bigg[\Big(\int_0^t\int_{\R}\left(f_n(X_{s-}+k(X_{s-},z))-f_n(X_{s-})-f(X_{s-}+k(X_{s-},z))+f(X_{s-})\right)\\&\hspace{9.8cm}\bigl(\tilde{\mu}(\cdot,\D z,\D s)- \tilde{\nu}(\D z) \,\D s\bigr)\Big)^2\bigg]\\
&\,=\,\widetilde{\mathbb{E}}_x\bigg[\int_0^t\int_{\R}\left(f_n(X_{s}+k(X_{s},z))-f_n(X_{s})-f(X_{s-}+k(X_{s},z))+f(X_{s})\right)^2\\&\hspace{13.05cm}\tilde{\nu}(\D z)\,\D s\bigg]
\,.
\end{align*}
Now, since \begin{align*}&|f_n(x+y)-f_n(x)-f(x+y)+f(x)|^2\\&\,\le\,
\left(\int_0^1\left(|\nabla f_n(x+ry)|+|\nabla f(x+ry)|\right)|y|\,\D r\right)^2\\
&\,\le\,4\lVert f\rVert_{1
}^2|y|^2\,,
\end{align*}
 the dominated convergence theorem  implies that, possibly passing to a subsequence, the third term on the left-hand side in \cref{e:sde} converges to $$\int_0^t\int_{\R^d}\left(f(X_{s-}+k(X_{s-},z))-f(X_{s-})\right)\bigl(\tilde\mu(\cdot,\D z,\D s)-\tilde\nu(\D z)\,\D s\bigr)\,,$$ which proves the desired result.

Finally, by following the above arguments, one easily sees that under \cref{eq:it2} the  assertion also holds   for any $f\in C_b^1(\R^d)$, which concludes the proof.
	\end{proof}

\subsection*{Proof of  \Cref{T1.1} (b)}
We now prove the main result of this article. We follow the approach from \cite{franke-periodic}.
Let $\beta\in C_b^{\varphi\psi}(\R^d)$ be a $\tau$-periodic solution to \cref{e:po-1} discussed above.
Recall that
either $\beta\in C_b^2(\R^d)$ if $c(x)\not\equiv0$,
or $m_{\varphi\psi}>1$
(resp.\ $\beta\in C_b^1(\R^d)$) if $c(x)\equiv0$ and \cref{eq:it}
(resp.\ \cref{eq:it2})
holds true,
as assumed in \Cref{T1.1} (b) (1), (2) and (3).
According to \Cref{Ito},  we can apply It\^{o}'s formula to the process $\{\beta(X_t)\}_{t\ge0}$.
Let us consider now the process $\{X_t-\bar {b^*}t-\beta(X_t)+\beta(X_0)\}_{t\ge0}$.  By combining \cref{SDE} and It\^{o}'s formula  we have that
\begin{align*}
&X_t-\bar {b^*} t-\beta(X_t)+\beta(X_0)\\&\,=\,  x+\int_0^{t}\left(b(X_{s
})+ \int_{B_1^c(0)}y\,\nu(X_{s
},\D y)- \bar {b^*}\right)\D s+\int_0^{t}\tilde\sigma(X_{s
})\,\D \tilde{W}_s\\&\ \ \ \ \ \ \, +\int_0^{t}\int_{\R}k(X_{s-
},z)\left(\tilde{\mu}(\cdot,
\D z , \D s)-\tilde{\nu}(\D z)\,\D s\right)\\
&\ \ \ \ \, \ \ -\int_0^t \bigl\langle\nabla \beta(X_{s
}),b(X_{s
})\bigr\rangle\,\D s-\int_0^t \bigl\langle\nabla \beta(X_{s
}),\tilde\sigma(X_{s
})\,\D \tilde W_s\bigr\rangle\\
&\ \ \ \ \ \ \, -\frac{1}{2}\sum_{i,j=1}^d\int_0^tc_{ij}(X_{s
})\partial_{ij}\beta(X_{s
})\ \D s\\
&\ \ \ \ \ \ \, - \int_0^t\int_{\R}\left(\beta(X_{s-}+k(X_{s-},z))-\beta(X_{s-})\right)\bigl(\tilde{\mu}(\cdot,\D z,\D s)-\tilde{\nu}(\D z)\,\D s\bigr)
\\
&\ \ \ \ \ \ \, - \int_0^t\int_{\R}\Big(\beta(X_{s}+k(X_{s},z))-\beta(X_{s})\\
&\qquad \qquad \qquad\quad  -\langle\nabla \beta(X_{s}),k(X_{s},z)\rangle\Ind_{\{u:|k(X_{s},u)|< 1\}}(z)\Big)\,\tilde{\nu}(\D z)\, \D s\\
&\,=\, x+\int_0^{t}\tilde\sigma(X_{s})\,\D \tilde{W}_s-\int_0^t \bigl\langle\nabla \beta(X_{s}),\tilde\sigma(X_{s})\,\D \tilde W_s\bigr\rangle \\
&\ \ \ \ \ \ \, +\int_0^{t}\int_{\R}k(X_{s-},z)\left(\tilde{\mu}(\cdot,\D z,\D s)-\tilde{\nu}(\D z)\, \D s\right)\\
&\ \ \ \ \ \ \,- \int_0^t\int_{\R}\left(\beta(X_{s-}+k(X_{s-},z))-\beta(X_{s-})\right)\bigl(\tilde{\mu}(\cdot,\D z,\D s)-\tilde{\nu}(\D z)\, \D s\bigr) \,,
\end{align*}
where we used
the fact
that $b^*(X_t)-\bar {b^*}=\mathcal{A}^b\beta(X_t)$ for any $t\ge0$.
	
	Clearly, $\{X_t-\bar {b^*} t-\beta(X_t)+\beta(X_0)\}_{t\ge0}$ is a special semimartingale, and from \cite[Proposition II.2.29]{jacod} we see again that for the truncation function we can take $h(x)=x$. Thus, the first characteristic of  $\{X_t-\bar {b^*} t-\beta(X_t)+\beta(X_0)\}_{t\ge0}$ vanishes (that is, $B_t\equiv0$, $t\ge0$), the third
	characteristic equals
	to
	\begin{align*}C^{ij}_t&\,=\,\int_0^t \Bigg(c_{ij}(X_s)+ \sum_{k,l=1}^dc_{kl}(X_s)\partial_k\beta_i(X_s)\partial_l\beta_j(X_s)\\&\hspace{1.6cm}-\sum_{l=1}^dc_{il}(X_s)\partial_l\beta_j(X_s)-\sum_{k=1}^dc_{kj}(X_s)\partial_k\beta_i(X_s)\Bigg)\D s\\
	&\,=\, \int_0^t\sum_{k,l=1}^{d}\left(\delta_{ki}-\partial_k\beta_i(X_s)\right)c_{kl}(X_s)\left(\delta_{lj}-\partial_l\beta_j(X_s)\right)\D s\,,\end{align*} for $t\ge0$ and $i,j=1,\dots,d,$ and the third modified characteristic reads
	\begin{align*}\widetilde{C}^{ij}_t\,=\,&C^{ij}_t+\int_0^t\int_{\R^d}y_iy_j\,\nu(X_s,\D y)\,\D s\\&+\int_0^t\int_{\R^d}\left(\bigl(\beta_i(X_s+y)-\beta_i(X_s)\bigr)\bigl(\beta_j(X_s+y)-\beta_j(X_s)\bigr)\right)\,\nu(X_s,\D y)\,\D s\\&-2\int_0^t\int_{\R^d}y_i\bigl(\beta_j(X_s+y)-\beta_j(X_s)\bigr)\,\nu(X_s,\D y)\,\D s\,,\end{align*} for $t\ge0$ and $i,j=1,\dots,d.$ Also, from \cite[Proposition II.2.17]{jacod} and \cite[Theorem 3.5]{rene-holder} we see that the second characteristic is $$ N(B,\D s)\,=\,\int_{\R^d}\Ind_{B}\left(y-\bigl(\beta(X_s+y)-\beta(X_s)\bigr)\right)\,\nu(X_s,\D y)\,\D s\,,\qquad B\in\mathcal{B}(\R^d)\,.$$

Consequently,
	 for any $x\in\R^d$,
		$$\{\varepsilon X_{\varepsilon^{-2}t}-\bar {b^*} \varepsilon^{-1}t-\varepsilon \beta(X_{\varepsilon^{-2}t})+\varepsilon \beta(X_0)\}_{t\geq0}\,,\qquad \varepsilon>0\,,$$ is
		a	$\mathbb{P}_{x}$- semimartingale (with respect to the natural filtration generated by $\process{X}$) whose (modified) characteristics (relative to $h(x)=x$) are given by
		\begin{align*}B^{\varepsilon,i}_t&\,=\,0\,,\\
		C^{\varepsilon,ij}_t&\,=\,
		\varepsilon^2\int_0^{\varepsilon^{-2}t}\sum_{k,l=1}^{d}\left(\delta_{ki}-\partial_k\beta_i(X_{s})\right)c_{kl}(X_{s})\left(\delta_{lj}-\partial_l\beta_j(X_{s})\right)\D s\,,\\
		N^{\varepsilon}(\D s,B)&\,=\,\frac{1}{\varepsilon^{2}}\int_{\R^d}\Ind_{B}\left(\varepsilon y-\varepsilon\bigl(\beta(X_{\varepsilon^{-2}s}+y)-\beta(X_{\varepsilon^{-2}s})\bigr)\right)\nu(X_{\varepsilon^{-2}s},\D y)\,\D s\,,\\
		\widetilde{C}^{\varepsilon,ij}_t&\,=\,C^{\varepsilon,ij}_t+\varepsilon^2\int_0^{\varepsilon^{-2}t}\int_{\R^d}y_iy_j\,\nu(X_{s},\D y)\,\D s\\&\ \ \ \ \ +\varepsilon^2\int_0^{\varepsilon^{-2}t}\int_{\R^d}\bigl(\beta_i(X_s+y)-\beta_i(X_s)\bigr)\\&
		\qquad\qquad\qquad\qquad\quad \bigl(\beta_j(X_s+y)-\beta_j(X_s)\bigr)\nu(X_s,\D y)\,\D s\\&\ \ \ \ \ -2\varepsilon^2\int_0^{\varepsilon^{-2}t}\int_{\R^d}y_i\bigl(\beta_j(X_s+y)-\beta_j(X_s)\bigr)\nu(X_s,\D y)\,\D s\,,
		\end{align*} $t\geq0$, $B\in\mathcal{B}(\R^{d})$, $i,j=1,\ldots,d$, (see \cite[Lemma 3.2 and Theorem 3.5]{rene-holder} and \cite[Proposition II.2.17]{jacod}).
		
		Further, observe that due to boundedness of $\beta(x)$, $\{\varepsilon X_{\varepsilon^{-2}t}-\bar {b^*} \varepsilon^{-1}t-\varepsilon \beta(X_{\varepsilon^{-2}t})+\varepsilon \beta(X_0)\}_{t\geq0}$ converges in the Skorohod space as $\varepsilon\to0$ if, and only if, $\{\varepsilon X_{\varepsilon^{-2}t}-\bar {b^*} \varepsilon^{-1}t\}_{t\geq0}$ converges, and if this is the case the limit is the same.
		Now, according to
		\cite[Theorem VIII.2.17]{jacod}, in order to prove the desired convergence it suffices to prove that
		\begin{equation*}\sup_{0\le s\le t}B^{\varepsilon,i}_s\,\xrightarrow[\varepsilon\to0]{\mathbb{P}_{x}}\, 0\,,
		\end{equation*} for all  $t\geq0$ and  $i=1,\ldots,d,$ which is trivially satisfied,
		\begin{equation}\label{ETP1}\int_0^{t}\int_{\R^{d}}g(y)N^{\varepsilon}(\D y,\D s)\,\xrightarrow[\varepsilon\to0]{\mathbb{P}_{x}}\,0\,,
		\end{equation} for all $t\geq0$ and  $g\in C_b(\R^{d})$
		vanishing in a neighborhood of the origin,
		and\begin{equation}\label{ETP2}\widetilde{C}^{\varepsilon,ij}_t\,\xrightarrow[\varepsilon\to0]{\mathbb{P}_{x}}\,t\Sigma^{ij}\,,\end{equation}
		for all  $t\geq0$ and  $i,j=1,\ldots,d$, where $\Sigma$ is given in \cref{ET1.2} and  $\xrightarrow[]{\mathbb{P}_{x}}$ stands for the convergence in probability.
		
		To prove the convergence in \cref{ETP2}, first observe that due to $\tau$-periodicity of all components we can replace $\process{X}$ by $\{\Pi(X_t)\}_{t\ge0}$, which is an ergodic Markov process (see  \Cref{SS2.1}).
		The assertion now follows as a direct consequence of the Birkhoff ergodic theorem.
		
		To prove the relation in \cref{ETP1} we proceed as follows. Fix  $g\in C_b(\R^{d})$  that vanishes on $B_\delta(0)$ for some $\delta>0$.
		Define now \begin{align*}F^{\varepsilon}(x)\,:=\,&\frac{1}{\varepsilon^{2}}\int_{\R^{d}}g\left(\varepsilon y-\varepsilon\bigl(\beta(x+y)-\beta(x)\bigr)\right)\nu(x,\D y)\\
		&-\frac{1}{\varepsilon^{2}}\int_{\mathbb{T}_\tau^d}\int_{\R^{d}}g\left(\varepsilon y-\varepsilon\bigl(\beta(z+y)-\beta(z)\bigr)\right)\nu(z,\D y)\,\pi(\D z)\,.\end{align*}
		Clearly, for any $\varepsilon>0$, $F^{\varepsilon}(x)$ is bounded, $\tau$-periodic, and  satisfies $F^{\varepsilon}(X_t)=
		F^{\varepsilon}\bigl(\Pi_{\tau}(X_t)\bigr)$ for $t\geq0$, and $$\int_{\mathbb{T}_\tau^d}F^{\varepsilon}(x)\,\pi(\D x)\,=\,0\,.$$
		Now, by the Markov property and exponential ergodicity of $\{\Pi_{\tau}(X_t)\}_{t\ge0}$, we have
		\begin{align*}&\mathbb{E}_{x}\left[\left(\int_0^{t}F^{\varepsilon}(X_{\varepsilon^{-2}s})\,\D s\right)^{2}\right]\\&
		\,=\,\mathbb{E}^{\tau}_{x_\tau}\left[\left(\int_0^{t}F^{\varepsilon}\bigl(\Pi_{\tau}(X_{\varepsilon^{-2}s})\bigr)\,\D s\right)^{2}\right] \nonumber\\&\,=\,2\int_0^{t}\int_0^{s}\mathbb{E}^{\tau}_{x_\tau}\left[F^{\varepsilon}\bigl(\Pi_{\tau}(X_{\varepsilon^{-2}s})\bigr)F^{\varepsilon}\bigl(\Pi_{\tau}(X_{\varepsilon^{-2}u})\bigr)\right]\D u\,\D s\\
		&\,=\,2\int_0^{t}\int_0^{s}\mathbb{E}^{\tau}_{x_\tau}\left[F^\varepsilon\bigl(\Pi_{\tau}(X_{\varepsilon^{-2}u})\bigr)P^{\tau}_{\varepsilon^{-2}(s-u)}F^{\varepsilon}\bigl(\Pi_{\tau}(X_{\varepsilon^{-2}u})\bigr)\right]\D u\,\D s\\
		&\,\le\, 2\Gamma\lVert F^{\varepsilon}\rVert^{2}_{\infty}\int_0^{t}\int_0^{s}\E^{-\gamma\varepsilon^{-2}(s-u)}\,\D u\,\D s\\&\le\frac{2\Gamma\varepsilon^2t}{\gamma}\lVert F^{\varepsilon}\rVert^{2}_{\infty}\\
		&\,\le\, \frac{8\Gamma\lVert g\rVert_\infty^2t}{\varepsilon^2\gamma}\sup_{x_\tau\in\mathbb{T}_\tau^d}\left\lvert\int_{\R^{d}}\Ind_{B_\delta^c}\left(\varepsilon y-\varepsilon\bigl(\beta(x+y)-\beta(x)\bigr)\right)\nu(x,\D y)\right\rvert^2\,.
		\end{align*}	
		Let $\varepsilon>0$ be such that $2\varepsilon\lVert \beta\rVert_\infty<\delta/2$. Then,
		\begin{align*}\mathbb{E}_{x}\left[\left(\int_0^{t}F^{\varepsilon}(X_{\varepsilon^{-2}s})\, \D s\right)^{2}\right]&
		\,\le\,	\frac{8\Gamma\lVert g\rVert_\infty^2t}{\varepsilon^2\gamma}\sup_{x_\tau\in\mathbb{T}_\tau^d}\left\lvert\int_{\R^{d}}\Ind_{B_{\delta/2}^c}\left(\varepsilon y\right)\nu(x,\D y)\right\rvert^2\\
		&\,=\,\frac{8\Gamma\lVert g\rVert_\infty^2t}{\varepsilon^2\gamma}\sup_{x_\tau\in\mathbb{T}_\tau^d}\left\lvert\nu(x,B_{\delta/2\varepsilon}^c)\right\rvert^2\,.
		\end{align*}	
		Now, since $$\frac{\delta^2}{4\varepsilon^2}\nu(x,B_{\delta/2\varepsilon}^c)\,\le\,\int_{B_{\delta/2\varepsilon}^c}\rvert y\rvert^2\,\nu(x,\D y)\,,$$ we have that
		\begin{align*}\mathbb{E}_{x}\left[\left(\int_0^{t}F^{\varepsilon}(X_{\varepsilon^{-2}s})\, \D s\right)^{2}\right]&
		\,\le\,	\frac{128\Gamma\lVert g\rVert_\infty^2\varepsilon^2t}{\gamma\delta^4}\sup_{x_\tau\in\mathbb{T}_\tau^d}\left(\int_{B_{\delta/2\varepsilon}^c}\rvert y\rvert^2\,\nu(x,\D y)\right)^2\,.
		\end{align*}
		Consequently, \begin{align*}&\left(\mathbb{E}_{x}\left[\left(\int_0^{t}\int_{\R^{d}}g(y)\,N^{\varepsilon}(\D y,\D s)\right)^{2}\right]\right)^{\frac{1}{2}}\\&\,\leq\,\left(\mathbb{E}_{x}\left[\left(\int_0^{t}F^{\varepsilon}(
		{X}_{\varepsilon^{-2}s})\,\D s\right)^{2}\right]\right)^{\frac{1}{2}}\\
		&\ \ \ \ \
		+\left(\mathbb{E}_{x}\left[\left(\frac{t}{\varepsilon^{2}}\int_{\mathbb{T}_\tau^d}\int_{\R^{d}}g\left(\varepsilon y-\varepsilon\bigl(\beta(z+y)-\beta(z)\bigr)\right)\,\nu(z,\D y)\,\pi(\D z)\right)^{2}\right]\right)^{\frac{1}{2}}\\&\,\leq\,
		\frac{8\sqrt{2}\Gamma^{1/2}\lVert g\rVert_\infty\varepsilon t^{1/2}}{\gamma^{1/2}\delta^2}\sup_{x_\tau\in\mathbb{T}_\tau^d}\int_{B_{\delta/2\varepsilon}^c}\rvert y\rvert^2\nu(x,\D y)\\&
		\quad \ \ +\frac{4\lVert g\rVert_{\infty}t}{\delta^{2}}\int_{\mathbb{T}_\tau^d}\int_{B_{\delta/2\varepsilon}^c}\frac{\delta^2}{4\varepsilon^2}\,\nu(z,\D y)\,\pi(\D z)\\
		&\,\le\, \left(\frac{8\sqrt{2}\Gamma^{1/2}\lVert g\rVert_\infty\varepsilon t^{1/2}}{\gamma^{1/2}\delta^2}
		+\frac{4\lVert g\rVert_{\infty}t}{\delta^{2}}\right)\sup_{x_\tau\in\mathbb{T}_\tau^d}\int_{B_{\delta/2\varepsilon}^c}\rvert y\rvert^2\,\nu(x,\D y)\,,
		\end{align*}
		which together with (\textbf{C3}) concludes the proof. \hfill\(\Box\)

\section{On Structural Properties of LTPs}\label{examples}
In this section, we
present sufficient conditions for LTPs  satisfying
strong Feller property, open-set irreducibility,
regularity
property of the semigroup, and regularity properties of the solution to the Poisson equation \cref{e:po-2}, respectively.
Several examples are also included.

\subsection*{Strong Feller Property} Let $\process{X}$ be a L\'evy-type process with symbol $q(x,\xi)$ and L\'evy triplet $(b(x),c(x),\nu(x,\D y))$.

\medskip

\begin{itemize}
	\item [(i)]

Let $\process{X}$
 be a diffusion process (that is,  $\nu(x,\D y)\equiv0$). According to \cite[Theorem  V.24.1]{rogersII}, it will be strong Feller  if $b(x)$ is  measurable,  $c(x)$  is continuous and positive definite, and
 there is a constant $\Lambda>0$ such that \begin{equation}\label{ER1.4}|c_{ij}(x)|+|b_i(x)|^2\,\le\,\Lambda(1+|x|^2)\,,\qquad x\in\R^d\,,\ i,j=1,\dots,d\,.\end{equation}

Let us also remark that
when  $\process{X}$ is a diffusion process generated with a second-order elliptic operator in divergence form \begin{equation}\label{E-DIV}
\mathcal{L}f(x)\, =\, \nabla\bigl(c(x)\cdot\nabla f(x)\bigr)
\end{equation} with
$c(x)$ bounded, measurable and uniformly elliptic, strong Feller property of  $\process{X}$ has been discussed in \cite{aron, nash, stroock}.

	\medskip
	
	\item[(ii)] Suppose that $(x,\xi)\mapsto q(x,\xi)$ is continuous,
 $b(x)$ is  continuous and bounded,  $c(x)$ is continuous, bounded and positive definite, and $x\mapsto\int_{B}(1\wedge|y|^2)\,\nu(x,\D y)$ is continuous and bounded for any $B\in\mathcal{B}(\R^d)$. Then, according to \cite[Theorems 3.23, 3.24 and 3.25]{rene-bjorn-jian} and \cite[Theorem 4.3 and its remark]{str}, $\process{X}$ is strong Feller.

	\medskip
	
	\item[(iii)] Recently, there are lots of developments on heat kernel (that is, the transition density function) estimates of Feller processes.  The reader is referred to \cite{chen-chen-wang, chen-hu-xie-zhang, chen-zhang, chen-zhang2, grz-sz, kim-lee, kim-sig-vond} and the references therein for more details. In particular, let
	$$ \mathcal{L}f(x)\,=\,\int_{\R^d}  \bigl(f(x+y)-f(x)-\langle\nabla f(x), y\rangle \Ind_{B_1(0)}(y)\bigr)\frac{ \kappa(x,y)}{|y|^{d+\alpha(x)}} \,\D y\,,$$
	where $\alpha:\R^d \to (0,2)$
	is a
	H\"older continuous
	function such that
	\begin{align*}
	& 0\,<\,\alpha_1\,\le\, \alpha(x)\, \le\, \alpha_2<2\,,  \qquad    x \in \R^d\,,\\
	& |\alpha(x)-\alpha(y)|\,\le\, c_1(|x-y|^{\beta_1}\wedge 1)\,,   \qquad    x,y\in \R^d\,,
	\end{align*}
	for  some constants $c_1>0$ and $\beta_1\in(0,1]$, and $\kappa:\R^d\times \R^d \to (0,\infty)$ is a measurable function satisfying
	\begin{align*}
	& \kappa(x,y)\,=\,\kappa(x,-y)\,, \qquad  x,y \in \R^d\,,\\
	&  0\,<\,\kappa_1\,\le\, \kappa(x,y)\,\le\, \kappa_2\,<\,\infty\,, \qquad    x,y\in \R^d\,,\\
	&  |\kappa(x,y)-\kappa(\bar x,y)|\,\le\, c_2(|x-\bar x|^{\beta_2}\wedge 1)\,,  \qquad    x,\bar x,y\in \R^d\,,
	\end{align*}
	for  some constants $c_2>0$ and $\beta_2 \in (0,1]$.  If
	$
	({\alpha_2}/{\alpha_1})-1<\bar\beta_0/{\alpha_2},
	$ with  $\bar\beta_0\in(0,\beta_0]\cap(0,{\alpha_2}/{2})$ and $\beta_0= \min\{\beta_1, \beta_2\}$, then, by \cite[Thereoms 1.1 and 1.3]{chen-chen-wang}, $(\mathcal{L}, C_c^\infty(\R^d))$ generates a LTP. Furthermore, by upper bounds as well as H\"{o}lder regularity and gradient estimates of the heat kernel (see  \cite[Thereoms 1.1 and 1.3, and Remark 1.4]{chen-chen-wang}),  this
associated  process  is strong Feller.
	
	\medskip
	
	\item[(iv)] Let $\process{X}$ and $\process{\tilde X}$ be LTPs with semigroups $\process{P}$ and $\process{\tilde P}$, and Feller generators $(\mathcal{A}^\infty,\mathcal{D}_{\mathcal{A}^\infty})$ and $(\tilde{\mathcal{A}}^\infty,\mathcal{D}_{\tilde{\mathcal{A}}^\infty})$, respectively.
	Suppose that $\process{X}$ is strong Feller. If $\mathcal{A}^\infty-\tilde {\mathcal{A}}^\infty$ is a bounded operator on $(B_b(\R^d), \lVert\cdot\rVert_\infty)$, then the formula
	$$P_tf\,=\,\tilde P_tf+\int_0^t  P_{s}(\mathcal{A}^\infty-\tilde {\mathcal{A}}^\infty)  \tilde P_{t-s}f\,\D s\,,\qquad f\in \mathcal{D}_{\mathcal{A}^\infty}\cap\mathcal{D}_{\tilde{\mathcal{A}}^\infty}\,,$$ implies that $\process{\tilde X}$ is also strong Feller. Namely, since both $\process{X}$ and $\process{\tilde X}$ are LTPs, the above relation holds for any $f\in C_c^\infty(\R^d)$. According to \cite[Lemma 1.1.1]{chung},
the boundedness of $\mathcal{A}^\infty-\tilde {\mathcal{A}}^\infty$ and
the dominated convergence theorem, it also holds for $f(x)=\Ind_{O}(x)$ for any open set $O\subseteq\R^d$. The claim now follows from Dynkin's monotone class theorem.
	The assertion above roughly asserts that a bounded  perturbation preserves the strong Feller property. Below is an example.
	
	Let
	\begin{align*}\tilde{\mathcal{L}}f(x)\,=\,&\int_{B_1(0)} \bigl(f(x+y)-f(x)-\langle\nabla f(x), y\rangle\bigr)\frac{ \kappa(x,y)}{|y|^{d+\alpha(x)}} \,\D y \\
	&+\int_{B_1^c(0)}\bigl(f(x+y)-f(x)\bigr)\,\nu(x,\D y)\,,\end{align*}
	where $\alpha(x)$ and $\kappa(x,y)$ satisfy all the assumptions in (iii), and $\nu(x,\D y)$ is such that
	
	\medskip
	
	\begin{itemize}
		\item [(a)] $\displaystyle\sup_{x\in\R^d}\nu(x,B_1(0))=0;$
		
		\medskip
		
		\item [(b)] $\displaystyle\sup_{x\in \R^d}\int_{\R^d}|y|^2\nu(x,\D y)<\infty$;
		
		\medskip
		
		\item[(c)] $\displaystyle f\mapsto \int_{\R^d}\bigl(f(\cdot+y)-f(\cdot)\bigr)\,\nu(\cdot,\D y)$ is an operator on $C_\infty(\R^d)$.
	\end{itemize}
	
	\medskip
	
	\noindent For example, one can take $\nu(x,\D y)=\frac{\gamma(x,y)}{|y|^{d+\delta}}\Ind_{B^c_1(0)}(y)\,\D y$ with $\delta>0$ and $\gamma(x,y)$ nonnegative, bounded and such that $x\mapsto \gamma(x,y)$ is continuous for almost every $y\in\R^d$.
	Further, let $\mathcal L$ be the operator  given in (iii). Then, \begin{align*}(\tilde {\mathcal L}-\mathcal L)f(x)\,=\,&\int_{B_1^c(0)}\bigl(f(x+y)-f(x)\bigr)\,\nu(x,\D y)\\
	&- \int_{B_1^c(0)}\bigl(f(x+y)-f(x)\bigr)\,\frac{ \kappa(x,y)}{|y|^{d+\alpha(x)}} \,\D y\end{align*} is bounded on $(B_b(\R^d), \lVert\cdot\rVert_\infty)$.
	By assumption, it is also bounded on   $(C_\infty(\R^d),$\linebreak$\lVert\cdot\rVert_\infty)$.
	Now, according to \cite[Lemma 1.28]{rene-bjorn-jian} and \cite[Proposition 2.1]{shioz},  $\tilde{\mathcal L}=\mathcal L+(\tilde {\mathcal L}-\mathcal L)$ generates a LTP.
	This, along with the assertion above and the strong Feller property of the process associated with  $\mathcal{L}$, yields  the strong Feller property of the process associated with  $\tilde{\mathcal{L}}$.
\end{itemize}

\medskip

We remark also  that the strong Feller property of LTPs has been discussed in \cite{strong}. In the special case when $\process{X}$ is given through \cref{SDE2},
the strong Feller property (and
the open-set irreducibility)  has been discussed in \cite{kwon} under the assumption that $\nu_Z(\R^m)<\infty$, and in \cite{masuda, masuda-err} for an arbitrary $\nu_Z(\D y)$, that is, an  arbitrary pure-jump L\'evy process $\process{Z}$. Observe that in both situations  non-degeneracy of $\Phi_2(x)\Phi'_2(x)$ has been assumed. In the case when $\Phi_3(x)\equiv\Phi_3\in\R^{d\times q}$  the problem has been considered in \cite{ari,liang-sch-wang, luo-wang}, and for non-constant (and non-degenerated) $\Phi_3(x)$ in \cite{liang-wang}.

\subsection*{Open-Set Irreducibility} Let $\process{X}$ be a L\'evy-type process with symbol $q(x,\xi)$ and L\'evy triplet $(b(x),c(x),\nu(x,\D y))$.

\medskip
	
	\begin{itemize}
		\item[(i)]
		According to \cite[Theorems V.20.1 and V.24.1]{rogersII} and \cite[Theorem 7.3.8]{durrett}, a diffusion process
		will be open-set irreducible (and strong Feller) if  $b(x)$ and $c(x)$ are locally H\"older continuous,
		$c(x)$ is positive definite,
		and \cref{ER1.4} holds true. Observe that
		\cref{ER1.4} trivially holds true in the periodic case.

		Also, when  $b(x)\in C_b^1(\R^d)$, $c(x)\in C_b^2(\R^d)$, $\partial_{ij}c_{kl}(x)$ is uniformly continuous for all $i,j,k,l=1,\dots,d$, and $c(x)$ is positive definite,
open-set irreducibility (and
 strong Feller property) of the process follows from the support theorem for diffusion processes, see  \cite[Lemma 6.1.1]{fried} and \cite[p.\ 517]{ikeda-watanabe}. For support theorem of jump processes one can refer to \cite{simon}.
		
		\medskip

		\item [(ii)] If $\process{X}$ is a diffusion process generated
 by a second-order elliptic operator in  divergence form \cref{E-DIV} with uniformly elliptic, bounded and measurable diffusion coefficient,
 open-set irreducibility (and
 strong Feller property) follows  from  the corresponding heat kernel estimates (see \cite{aron,nash,stroock}).
		
		The  diffusion processes with jumps
		or pure jump process considered in \cite{chen-chen-wang, chen-hu-xie-zhang, chen-zhang, chen-zhang2, grz-sz, kim-lee, kim-sig-vond} are also open-set irreducible, which is a direct consequence of obtained lower bounds of heat kernel.

		\medskip
		
		\item[(iii)] Let $\mathcal{L}$ and $\tilde{\mathcal{L}}$ be the operators from (iv)
 in the discussion on
 the strong Feller property.
		According to \cite[Thereom 1.3]{chen-chen-wang}, the LTP corresponding to $\mathcal{L}$  is open-set irreducible. Further, observe  that $$\sup_{x\in \R^d} \int_{B_1^c(0)} \frac{\kappa(x,y)}{|y|^{d+\alpha(x)}}\,\D y\,<\,\infty\,.$$
		Thus, by \cite[Lemma 3.1]{barlow} and \cite[Lemma 3.6]{barlow2}, the process  associated with the operator $\tilde{\mathcal{L}}$ is also open-set irreducible.
			\end{itemize}

\medskip

	For  open-set irreducibility  of LTPs of the form \cref{SDE1} we refer the reader to  \cite{ari}, \cite{kwon} and \cite{masuda,masuda-err}.

In the following proposition, which slightly generalizes \cite[Lemma 2]{horie-inuzuka-tanaka}, we show that a LTP will be open-set irreducible if the corresponding L\'evy measure shows enough jump activity. First, recall that a function $f:\R^d\to\R$ is said to be \textit{lower semi-continuous} if \begin{equation*}\liminf_{y\to x}f(y)\,\ge\, f(x)\,,\qquad x\in\R^d\,.
\end{equation*}
\begin{proposition} \label{p2.1}
	The process $\process{X}$ will be open-set irreducible if
	there are constants $R>r\ge0$ such that
	
	\medskip
	
	\begin{itemize}
		\item [(i)] $\inf_{x\in K}\nu(x,O)>0$
		for every non-empty open set $O\subseteq B_R(0)\setminus B_r(0)$, and every non-empty compact set $K\subset \R^d$;
		
		\medskip
		
		\item[(ii)] the function $x\mapsto \int_{\R^d}f(y+x)\nu(x,\D y)$ is lower semi-continuous for every non-negative lower semi-continuous function $f:\R^d\to\R.$
	\end{itemize}
\end{proposition}
\begin{proof}
	Let $x\in\R^d$ and $\rho>0$ be arbitrary,  and let $f\in C_c^{\infty}(\R^d)$
be such that
 $0\leq f\leq 1$ and ${\rm supp}\, f\subset B_\rho(x)$. By assumption, $$\lim_{t\to0}\left\|\frac{P_t f-f}{t}-\mathcal{A}^\infty f\right\|_\infty\,=\,0\,.$$ In particular, for any $B\subseteq  B^c_\rho(x)$,
	\begin{align*}
	\liminf_{t\to0}\inf_{y\in B}\frac{\mathbb{P}_{y}\bigl(X_t\in  B_\rho(x)\bigr)}{t}&\,\geq\, \liminf_{t\to0}\inf_{y\in B}\frac{P_tf(y)}{t}\\&\,=\,\liminf_{t\to0}\inf_{y\in B}\left|\frac{P_tf(y)}{t}-\mathcal{A}^\infty f(y)+\mathcal{A}^\infty f(y)\right|\\&\,=\,\inf_{y\in B}|\mathcal{A}^\infty f(y)|\\
	&\,=\,\inf_{y\in B}\int_{\R^d}f(z+y)\nu(y,\D  z)\,.\end{align*}
	Further, let $0<\varepsilon<\rho$ be arbitrary, and let
$0\le f_\varepsilon\in C_c^{\infty}(\R^d)$ be
such that
	$$f_\varepsilon(y)\,=\,\left\{
	\begin{array}{ll}
	1, & y\in B_{\rho-\varepsilon}(x) \\
	0, & y\in B^c_{\rho}(x)\,.
	\end{array}
	\right.$$
	Then, for any $y\in B^c_{\rho}(x)$ we have that
	\begin{align*}\mathcal{A}^\infty f_\varepsilon(y)&\,=\,\int_{\R^d}f_\varepsilon(z+y)\,\nu(y,\D z)\\&\,\geq\,\nu\bigl(y,\bigl(B_R(0)\setminus B_r(0)\bigr)\cap B_{\rho-\varepsilon}(x-y)\bigr)\,.
	\end{align*}
	
	Next, take $x,y\in\R^d$ such that $r<|x-y|<R$, and pick $\varepsilon,\rho>0$ such that $\varepsilon<\rho$ and $r+2\rho<|x-y|<R-2\rho$. Then we have
	\begin{align*}
	\liminf_{t\to0}\inf_{z\in B_\rho(x)}\frac{\mathbb{P}_{z}\bigl(X_t\in  B_\rho(y)\bigr)}{t}&\,\ge\, \inf_{z\in B_\rho(x)}\nu\bigl(z,\bigl(B_R(0)\setminus B_r(0)\bigr)\cap B_{\rho-\varepsilon}(y-z)\bigr)\\
&\,=\, \inf_{z\in B_\rho(x)}\nu\bigl(z, B_{\rho-\varepsilon}(y-z)\bigr)\,.\end{align*}
	Assume now that $\inf_{z\in B_\rho(x)}\nu\bigl(z, B_{\rho-\varepsilon}(y-z)\bigr)=0.$ Then there is a sequence $\{z_n\}_{n\in\N}\subset B_\rho(x)$ converging to $z_0\in\bar{B}_\rho(x)$ such that
	\begin{equation*}
	\liminf_{n\to\infty}\nu\bigl(z_n, B_{\rho-\varepsilon}(y-z_n)\bigr)\,=\,\liminf_{n\to\infty}\int_{\R^d}\Ind_{B_{\rho-\varepsilon}(y)}(u+z_n)\,\nu(z_n,\D u)\,=\,0\,.
	\end{equation*}
	However, since $z\mapsto \Ind_{B_{\rho-\varepsilon}(y)}(z)$ is a lower semi-continuous function, we have that
	\begin{equation*}
	\liminf_{n\to\infty}\int_{\R^d}\Ind_{B_{\rho-\varepsilon}(y)}(u+z_n)\,\nu(z_n,\D u)\,=\,\nu\bigl(z_0, B_{\rho-\varepsilon}(y-z_0)\bigr)\,>\,0\,,
	\end{equation*}
	which is in contradiction with the above assumption. Hence, there is  $t_*=
t_*(x,y,\rho,\varepsilon)>0$ such that
	\begin{equation*}
	\mathbb{P}_{z}\bigl(X_t\in B_{\rho}(y)\bigr)\,>\,0\,,\qquad z\in B_\rho(x)\,,\ t\in(0,t_*]\,.
	\end{equation*}
	
	Fix now $\varepsilon,\rho>0$ such that $\varepsilon<\rho$  and $4\rho<R-r$. From the previous discussion it follows that for any
	$x,y\in\R^d$ with  $r+2\rho<|x-y|<R-2\rho$, there is  $t_{**}=
t_{**}(x,y,\rho,\varepsilon)>0$ such that \begin{equation*}
	\mathbb{P}_{z}\bigl(X_t\in B_{\rho}(y)\bigr)\,>\,0\,,\qquad z\in B_\rho(x)\,,\ t\in(0,t_{**}]\,.
	\end{equation*}
	The assertion now follows by employing the Chapman-Kolmogorov equation.		\end{proof}

   Observe that in \Cref{p2.1} we require that $\nu(x,\D y)$ is not singular with respect to the $d$-dimensional Lebesgue measure. However, there are many  interesting open-set irreducible LTPs which do not meet this property. For example, let $\process{X}$ be a solution to \cref{SDE1} with  $n=d+1$,  $\Phi(x)=(\bar\Phi(x),\mathbb{I}_{d})$ and $Y_t=(t,B_t,Z_t)'$, $t\ge0$, where $\bar\Phi:\R^d\to\R^d$, $\mathbb{I}_{d}$ is the $d\times d$-identity matrix, $\process{B}$ is a $d_1$-dimensional Brownian motion with $1\le d_1<d$, and $\process{Z}$ is a $(d-d_1)$-dimensional rotationally invariant  $\alpha$-stable L\'evy process (independent of  $\process{B}$) with $\alpha\in(0,2)$.
Clearly, in this case
	 the L\'{e}vy measure  is $(d-d_1)$-dimensional. Thus, \Cref{p2.1}  cannot be applied to the process $\process{X}$.
		However, open-set irreducibility of $\process{X}$  may be concluded by employing the time-changed idea as in \cite{WZ}. Namely, the Girsanov transformation implies open-set irreducibility of a solution to \cref{SDE1} with  $\Phi_0(x)$ similar to $\Phi(x)$ defined  above and $\bar Y_t=(t,B_t,\bar B_t)'$, $t\ge0$, where  $\process{B}$ is also as above, and $\process{\bar B}$ is a $(d-d_1)$-dimensional Brownian motion (independent of  $\process{B}$).
	 With this at hand, and  following the approach in \cite{WZ} (the time-changed idea combined with approximation argument),  we conclude open-set irreducibility of  $\{X_t\}_{t\ge0}$.
	An alternative approach is based on the Levi's method from PDE theory. Namely, since the transition function of the process $\{(B_t,  Z_t)\}_{t\ge0}$ enjoys the product form with Gaussian estimates and two-sided heat kernel estimates for rotationally invariant $\alpha$-stable processes, one may follow the argument from \cite{chen-zhang2} to get two-sided heat kernel estimates for  $\{X_t\}_{t\ge0}$. When $\alpha\in (0,1)$ we may need to additionally assume that $\Phi(x)$ is H\"{o}lder continuous.
	
	Let us also remark that open-set irreducibility (and strong Feller property) of
	 a solution to \cref{SDE1} with $\Phi(x)=(\bar\Phi(x),\mathbb{I}_{d})$ and
	$Y_t=(t,Z^1_t,\dots,Z^{d}_t)'$, $t\ge0$,  where $\bar\Phi:\R^d\to\R^d$, and $\{Z^i_t\}_{t\ge0}$, $i=1,\dots,d,$  are mutually independent  one-dimensional symmetric  $\alpha$-stable L\'{e}vy processes with $\alpha\in(1,2)$, has  been deduced in \cite[Theorem 3.1(iv)]{aris}. Note that in this case the L\'evy measure again does not satisfy (i) in \Cref{p2.1}.

\subsection*{Regularity Property of the Semigroup, and Regularity Properties of the Solution to \cref{e:po-2}} Let $\process{X}$ be a L\'evy-type process with $\tau$-periodic   L\'evy triplet $(b(x),c(x),\nu(x,\D y))$.

\medskip

\begin{itemize}
	\item[(i)] (Diffusion processes) Let $\varepsilon\in(0,1)$, and  let $\process{X}$ be a diffusion process with  coefficients  $b\in C_b^\varepsilon(\R^d)$, $c\in C_b^{1+\varepsilon}(\R^d)$,  and $c(x)$  being also positive definite. Then,
	({\bf C4})(i) with arbitrary $t_0>0$ and $\psi(r)=r^\varepsilon$ follows from \cite[the proof of Lemma 2.3]{tomisaki}. Also,
	a straightforward adaptation of  \cite[Theorem 2.1]{tomisaki}, together with \cite[Chapter 4.8]{jacobI} and \cite[Proposition 4.2]{pang-sand}, implies  ({\bf C4})(ii) with $\varphi(r)=r^2$. Then, the conclusion of Theorem \ref{T1.1} holds.
	
	\medskip
	
	\item[(ii)] (Diffusion processes with jumps) Let $\varepsilon\in(0,1)$. Assume that $b(x)$ and $c(x)$ are as in (i), and that $\nu(x,\D y)$ satisfies
	
	\medskip
	
	\begin{itemize}
		\item [(a)] $\displaystyle\sup_{x\in\R^d}\int_{B_1(0)}|z|^{1+\varepsilon}\,\nu(x,\D z)<\infty;$
		
		\medskip
		
		\item[(b)] $\displaystyle \lim_{\epsilon\to0}\sup_{x\in\R^d}\int_{B_\epsilon(0)}|z|^{1+\varepsilon}\,\nu(x,\D z)=0;$
		
		\medskip
		
		\item [(c)] $\displaystyle \lim_{R\to\infty}\sup_{x\in\R^d}\int_{B_R^c(0)}|z|^{1+\varepsilon}\,\nu(x,\D z)=0;$
		
		\medskip
		
		\item [(d)] $\displaystyle\sup_{x,y\in\R^d}|x-y|^{-\varepsilon}\int_{\R^d}\bigl(1\wedge|z|^{1+\varepsilon}\bigr)\,|\nu(x,\D z)-\nu(y,\D z)|<\infty$.
	\end{itemize}
	
	\medskip
	
	\noindent Here, $|\mu(\D z)|$ stands for the total variation measure of a signed measure $\mu(\D z)$. Then, ({\bf C4})(ii) with $\varphi(r)=r^2$ follows again from \cite[Theorem 2.1]{tomisaki}, together with \cite[Chapter 4.8]{jacobI} and \cite[Proposition 4.2]{pang-sand}.
	
	Let us give sufficient conditions that $\process{X}$ also satisfies ({\bf C4})(i).  Denote by $\process{P}$ the  semigroup of $\process{X}$, and
	let $\process{\tilde{P}}$ be the semigroup of the diffusion process with coefficients  $b(x)$ and $c(x)$. Also, denote by $(\mathcal{A}^\infty,\mathcal{D}_{\mathcal{A}^\infty})$
	and $(\tilde{\mathcal{A}}^\infty,\mathcal{D}_{\tilde{\mathcal{A}}^\infty})$ the corresponding $C_\infty$-generators, respectively. Then,
	$$P_tf\,=\,\tilde P_tf+\int_0^t  \tilde P_{s}(\mathcal{A}^\infty- \tilde{\mathcal{A}}^\infty)   P_{t-s}f\,\D s\,,\qquad f\in \mathcal{D}_{\mathcal{A}^\infty}\cap\mathcal{D}_{\tilde{\mathcal{A}}^\infty}\,.$$ Since both processes are LTPs, the above relation holds for any $f\in C_c^\infty(\R^d)$.  Assume next that $\mathcal{A}^\infty-\tilde{\mathcal{A}}^\infty$ is a bounded operator on $(B_b(\R^d),\lVert\cdot\rVert_\infty)$. Then,  according to \cite[Lemma 1.1.1]{chung},
the boundedness of $\mathcal{A}^\infty-\tilde {\mathcal{A}}^\infty$ and the dominated convergence theorem, the above relation  holds for $f(x)=\Ind_{O}(x)$ for any open set $O\subseteq\R^d$. Thus, it also holds for any $f\in B_b(\R^d).$ Recall also that $P_tf\in C_b(\R^d)$ for every $f\in C_b(\R^d)$ and every $t\ge0$ (see \cite[Corollary 3.4]{rene-conserv}). Now, according to (i),  there is
 a measurable function $C_\varepsilon:(0,\infty)\to(0,\infty)$ such that $\int_0^tC_\varepsilon(s)\,\D s <\infty$  and $\lVert \tilde P_t f\rVert_\varepsilon\le  C_\varepsilon(t)\lVert f\rVert_\infty$ for all $t>0$ and  all $\tau$-periodic $f\in C_b(\R^d)$. Thus,
	for fixed  $\tau$-periodic $f\in C_b(\R{^d})$, $P_tf\in C_b^\varepsilon(\R^d)$ and
	\begin{equation*}
	\lVert  P_t f\rVert_\varepsilon\,\le\, \lVert \tilde P_tf\rVert_\varepsilon+\int_0^t  \lVert\tilde P_{s}(\mathcal{A}^\infty- \tilde{\mathcal{A}}^\infty)   P_{t-s}f\rVert_\varepsilon\,\D s\\
	\,\le\, \bar C_\varepsilon(t) \lVert f \rVert_\infty\,,
	\end{equation*}
	where $\bar C_\varepsilon(t)=C_\varepsilon(t)+\lVert\mathcal{A}^\infty- \tilde{\mathcal{A}}^\infty\rVert \int_0^t C_\varepsilon(s)\,\D s.$ Also,
	$$\int_0^t\bar C(s)\,\D s\,\le\,(1+t\lVert\mathcal{A}^\infty- \tilde{\mathcal{A}}^\infty\rVert)\int_0^t C(s)\,\D s\,,\qquad t>0\,,$$ where $\lVert\mathcal{A}^\infty- \tilde{\mathcal{A}}^\infty\rVert$ stands for the operator norm of $\mathcal{A}^\infty- \tilde{\mathcal{A}}^\infty$. Thus, $\process{X}$ satisfies    (\textbf{C4})(i) with $\psi(r)=r^\varepsilon$.
 Therefore, if additionally  $\int_{B_1^c(0)} y\,\nu(\cdot,\D y)\in C_b^\varepsilon(\R^d)$,  the conclusion of \Cref{T1.1} holds true.
\medskip

\item[(iii)] (Pure-jump LTPs)
In the pure jump case, sufficient conditions for  (\textbf{C4})(i)  are given in \cite[Theorem 1.1]{liang-wang2}.  Also,  when the underlying process is given as a solution to an SDE of the form \cref{SDE2}, we   refer to \cite{liang-sch-wang, liang-wang, luo-wang} and the references therein.

To construct an example satisfying (\textbf{C4})(ii), we can again employ a perturbation method.	 Let $\process{X}$ and $\process{\tilde X}$ be LTPs with semigroups $\process{P}$ and $\process{\tilde P}$, and $B_b$-generators $(\mathcal{A}^b,\mathcal{D}_{\mathcal{A}^b})$ and $(\tilde{\mathcal{A}}^b,\mathcal{D}_{\tilde{\mathcal{A}}^b})$, respectively.
Assume that $\mathcal{A}^b$ satisfies (\textbf{C4})(ii) for some H\"{o}lder exponents $\psi(r)$ and $\varphi(r)$.
Further, assume that $\mathcal{A}^b-\tilde{\mathcal{A}}^b$ is a bounded operator on $(B_b(\R^d),\lVert\cdot\rVert_\infty)$, and that $(\mathcal{A}^b-\tilde{\mathcal{A}}^b)f\in C_b(\R^d)$ for every $f\in C_b(\R^d)$.
Then,
$$\tilde P_tf\,=\, P_tf+\int_0^t  P_{s}(\mathcal{A}^b-\tilde {\mathcal{A}}^b)  \tilde P_{t-s}f\,\D s\,,\qquad f\in \mathcal{D}_{\mathcal{A}^b}\cap\mathcal{D}_{\tilde{\mathcal{A}}^b}\,.$$ Similarly as before, the above relation holds for all $f\in B_b(\R^d)$.
Thus, for any $\lambda>0$ and any $\tau$-periodic $f\in C_b(\R^d)$, $$\tilde{R}^\tau_\lambda f_\tau\,=\,R^\tau_\lambda f_\tau+R^\tau_\lambda (\mathcal{A}^b-\tilde {\mathcal{A}}^b)\tilde{R}^\tau_\lambda f_\tau\,.$$
Assume now that $\{\tilde P_t\}_{t\ge0}$ satisfies  (\textbf{C4})(i) with $\psi(r)$, and that $(\mathcal{A}^b-\tilde{\mathcal{A}}^b)f\in C^\psi_b(\R^d)$ for every $f\in C^\psi_b(\R^d)$. Then, according to the proof of  \Cref{T1.1} (a), for any $\tau$-periodic $f\in C_b(\R^d)$ with $\int_{\mathbb{T}^d_\tau}f_\tau(x)\,\pi(\D x)=0$, $\tilde{R}^\tau_\lambda f\in C^\psi(\mathbb{T}^d_\tau)$ and so $(\mathcal{A}^b-\tilde{\mathcal{A}}^b)\tilde{R}^\tau_\lambda f_\tau \in C^\psi(\mathbb{T}^d_\tau)$. Hence, for any $\tau$-periodic $f\in C_b^\psi(\R^d)$, $\tilde{R}^\tau_\lambda f_\tau\in C^{\varphi\psi}(\mathbb{T}^d_\tau),$
that is, the corresponding $\tau$-periodic extension is a  solution to \cref{e:po-2}. Finally, uniqueness follows from the fact that
any solution $u(x)$ to \cref{e:po-2}  must have the representation $\int_0^\infty e^{-\lambda t} \tilde P_t f\,\D t,$ since $u=(\lambda -\tilde {\mathcal{A}}^b)^{-1}f.$

Below we give concrete examples of  LTPs $\process{X}$ and $\process{\tilde X}$ satisfying the above assumptions.
Let $\varphi:(0,\infty)\to(0,\infty)$ be  increasing, and such that $\varphi(1)=1$
and

\medskip

\begin{itemize}
	\item [(a)] there are $0<\underline{\alpha}\le\overline{\alpha}<1$, $\underline{\kappa}\in(0,1]$ and $\overline{\kappa}\in[1,\infty)$, such that \begin{equation}
	\label{scal}
	\underline{\kappa}\,\lambda^{2\underline{\alpha}}\varphi(r)\,\le\,\varphi(\lambda r)\,\le\,\overline{\kappa}\,\lambda^{2\overline{\alpha}}\varphi(r)\,,\qquad \lambda\ge1\,,\ r\in(0,1]\,;\end{equation}
	
	\medskip
	
	\item[(b)] $\displaystyle \int_1^\infty \frac{1}{r \varphi(r)}\,\D r<\infty.$

	\end{itemize}

\medskip

\noindent Then,
by (a), $\lim_{r\to0}\varphi(r)=0$,
and so
$\varphi(r)$ is a H\"{o}lder exponent with
$[m_\varphi,M_\varphi]\subseteq[2\underline{\alpha},2\overline{\alpha}]\subset(0,2)$. Further, let $n:\R^d\setminus\{0\}\to[\underline\Gamma,\overline\Gamma]$, with $0\le\underline\Gamma\le\overline{\Gamma}<\infty$, be measurable. Then,
thanks to (a) and (b),
 $$\nu_0(\D y)\,:=\,\frac{n(y)}{\varphi(|y|)|y|^{d}}\,\D y$$
is a L\'evy measure.
Denote the L\'evy process generated by the L\'evy triplet $(0,0,\nu_0(\D y))$ by $\process{X}$.
Also, let  $(\mathcal{A}^b,\mathcal{D}_{\mathcal{A}^b})$ be the corresponding $B_b$-genera- tor. Then, $\process{X}$ satisfies (\textbf{C4})(ii) for any H\"{o}lder exponent $\psi(r)$ such that $[m_\psi,M_\psi]\subset(0,1)$ and $[m_{\varphi\psi},M_{\varphi\psi}]\cap \N=\emptyset$. Namely, since $\process{X}$ has $\tau$-periodic (actually constant) coefficients,  the corresponding projection (with respect to $\Pi_{\tau}(x)$) on $\mathbb{T}_\tau^d$ is again a strong Markov process. Moreover, according to \cite[Proposition 2.2]{jian} (see also \cite[Theorem 2.1]{kno-sci}) and \Cref{p2.1},
 it is also strong Feller and open-set irreducible so it satisfies \cref{eq:erg}. Now, for any $\lambda>0$ and any $\tau$-periodic $f\in B_b(\R^d)$, we see as before that the $\tau$-periodic extension $u_{\lambda,f}(x)$ of $R_\lambda^\tau f_\tau(x)$ solves $\lambda u_{\lambda,f}-\mathcal{A}^bu_{\lambda,f}=f$. If $f\in C_b^\psi(\R^d)$ for some  H\"{o}lder exponent $\psi(r)$ such that $[m_\psi,M_\psi]\subset(0,1)$, then, according to \cite[Proposition 2.2]{jian} and the proof of  \cite[Propositions 3.5 and 3.6]{kassmann2}, $u_{\lambda,f}\in C_b^{\varphi\psi}(\R^d)$ provided $[m_{\varphi\psi},M_{\varphi\psi}]\cap \N=\emptyset$. Let us remark  that in the proofs of \cite[Propositions 3.5 and 3.6]{kassmann2} the authors require the scaling property \cref{scal} of $\varphi(r)$ for all $r\in(0,\infty)$, and the additional assumptions that $n(y)\equiv c$ for some  $c>0$ and that $\phi(r)=\varphi(r^{-1/2})^{-1}$ is a Bernstein function, that is, $(-1)^n\phi^{(n)}(r)\le 0$ for every $n\in\N_0$. They essentially use this property in order to apply \cite[Corollary 3.2]{kassmann2} via the regularity of semigroups associated with subordinated Brownian motions. However, the statement of this corollary has been proved in \cite[Proposition 2.2]{jian} under the scaling condition in \cref{scal}.
  Finally, uniqueness follows from (a straightforward adaptation of) \cite[Proposition 3.2]{priola} (by taking $b(x)\equiv0$).
A typical example of the function $\varphi(r)$ satisfying the above assumptions is given by $\varphi(r)=r^\alpha \log^\beta(\E-1+r^{-1})$ with $\alpha\in (0,2)$ and $\beta\in \R$. According to \cite[Proposition 2.2]{jian}, there is $c>0$ such that for all $t\in (0,1]$ and $f\in B_b(\R^d)$,
$\|\nabla P_tf\|_\infty\le c(\varphi^{-1}(t))^{-1}.$ Therefore, we have that

\medskip

\begin{itemize}
	\item [(1)] if $\alpha\in (1,2)$ or $\alpha=1$ and $\beta<-\alpha$, then {\bf(C4)}(i) is satisfied and Theorem \ref{T1.1} (b)(2) holds with $\psi(r)=r^{\theta_1}\log ^{\theta_2} (1+r^{-1})$ for any $\theta_1\in (0,1)$ and $\theta_2\in \R$;
	
	\medskip
	
	\item[(2)] if $\alpha\in (0,1)$, then {\bf(C4)}(i) is satisfied with $\psi(r)=r^{\theta_1}\log ^{\theta_2} (1+r^{-1})$ for any $\theta_1\in (0,\alpha)$ and $\theta_2\in \R$, and Theorem \ref{T1.1} (b)(3) holds.
\end{itemize}
Also, as we have commented above,  $\process{X}$ satisfies (\textbf{C4})(ii) if $\theta_1$ is such that $\alpha+\theta_1\notin\N$.

	Further, let $\process{\tilde X}$ be a LTP
generated by $(0,0,\nu(x,\D y))$ with $$\nu(x,\D y)\,=\Ind_{B_1(0)}(y)\,\nu_0(\D y)  + \frac{\gamma(x,y)}{\tilde{\varphi}(|y|)|y|^{d}}\Ind_{B_1^c(0)}(y)\,\D y\,,$$ where
$\tilde{\varphi}:[1,\infty)\to(0,\infty)$ satisfies that $\int_1^\infty \frac{1}{r\tilde{\varphi}(r)}\,\D r<\infty$,  and $\gamma(x,y)$ is non-negative, bounded and such that $x\mapsto \gamma(x,y)$ is continuous for almost every $y\in\R^d$ on $B_1^c(0)$ (see (iv)
in the discussion on the strong Feller property above that this L\'evy kernel generates a LTP). Denote by
$(\mathcal{\tilde A}^b,\mathcal{D}_{\mathcal{\tilde A}^b})$ the corresponding $B_b$-generator. It is easy to see that   $\mathcal{A}^b-\tilde{\mathcal{A}}^b$ is  bounded  on $(B_b(\R^d),\lVert\cdot\rVert_\infty)$, and  $(\mathcal{A}^b-\tilde{\mathcal{A}}^b)f\in C_b(\R^d)$ for every $f\in C_b(\R^d)$.
 Furthermore, $(\mathcal{A}^b-\tilde{\mathcal{A}}^b)f\in C^\psi_b(\R^d)$ for every $f\in C^\psi_b(\R^d)$ if
 we additionally assume that
 for almost all $y\in\R^d$ on $B_1^c(0)$, $x\mapsto\gamma(x,y)$ is of class $C_b^\psi(\R^d)$. With these at hand, we can follow the argument in (ii) to check that  {\bf(C4)} is satisfied, and so the conclusion of  \cref{T1.1} holds.

\end{itemize}

\section*{Acknowledgement}
  Financial support through the \textit{Alexander-von-Humboldt Foundation} and \textit{Croatian Science Foundation} under project 8958  (for N. Sandri\'c),
   Croatian Science Foundation under project 8958 (for I. Valenti\'c),
    and the National
Natural Science Foundation of China (No.\ 11831014), the Program for Probability and Statistics: Theory and Application (No.\ IRTL1704) and the Program for Innovative Research Team in Science and Technology in Fujian Province University (IRTSTFJ) (for J. Wang)
are  gratefully acknowledged. We also thank the anonymous referees for the helpful comments that have led to significant improvements of the results in the article.

\def\cprime{$'$} \def\cprime{$'$} \def\cprime{$'$}


\begin{thebibliography}{10}
	\providecommand{\bysame}{\leavevmode\hbox to3em{\hrulefill}\thinspace}
	\providecommand{\MR}{\relax\ifhmode\unskip\space\fi MR }
	\providecommand{\MRhref}[2]{%
		\href{http://www.ams.org/mathscinet-getitem?mr=#1}{#2}
	}
	\providecommand{\href}[2]{#2}
	\providebibliographyfont{name}{}%
	\providebibliographyfont{lastname}{}%
	\providebibliographyfont{title}{\emph}%
	\providebibliographyfont{jtitle}{\btxtitlefont}%
	\providebibliographyfont{etal}{}%
	\providebibliographyfont{journal}{}%
	\providebibliographyfont{volume}{\textbf}%
	\providebibliographyfont{ISBN}{\MakeUppercase}%
	\providebibliographyfont{ISSN}{\MakeUppercase}%
	\providebibliographyfont{url}{\url}%
	\providebibliographyfont{numeral}{}%
	\providecommand\btxprintamslanguage[1]{\ (#1)}
	\expandafter\btxselectlanguage\expandafter {\btxfallbacklanguage}
	
	\expandafter\btxselectlanguage\expandafter {\btxfallbacklanguage}


	\bibitem{aris}
A.~Arapostathis, G.~Pang, and N.~Sandri\'{c}.
\btxjtitlefont {\btxifchangecase
		{Ergodicity of a {L}\'{e}vy-driven {SDE} arising from multiclass
many-server queues,}} \btxjournalfont {Ann. Appl. Probab.} \btxvolumefont {29} (2019), \btxnumbershort {.}~2,
1070--1126.

	\bibitem {ari2}
	\btxnamefont {M.~\btxlastnamefont {Arisawa}}, \btxjtitlefont {\btxifchangecase
		{Quasi-periodic homogenizations for second-order
			{H}amilton-{J}acobi-{B}ellman equations}{Quasi-periodic homogenizations for
			second-order {H}amilton-{J}acobi-{B}ellman equations}}, \btxjournalfont {Adv.
		Math. Sci. Appl.} \btxvolumefont {11} (2001), \btxnumbershort {.}~1,
	465--480.
	
	\bibitem {ari}
	\bysame, \btxjtitlefont {\btxifchangecase {Homogenization of a class of
			integro-differential equations with {L}\'{e}vy operators}{Homogenization of a
			class of integro-differential equations with {L}\'{e}vy operators}},
	\btxjournalfont {Comm. Partial Differential Equations} \btxvolumefont {34}
	(2009), \btxnumbershort {.}~7-9, 617--624.
	
	\bibitem {ari3}
	\bysame, \btxjtitlefont {\btxifchangecase {Homogenizations of
			integro-differential equations with {L}\'{e}vy operators with asymmetric and
			degenerate densities}{Homogenizations of integro-differential equations with
			{L}\'{e}vy operators with asymmetric and degenerate densities}},
	\btxjournalfont {Proc. Roy. Soc. Edinburgh Sect. A} \btxvolumefont {142}
	(2012), \btxnumbershort {.}~5, 917--943.
	
	\bibitem {aron}
	\btxnamefont {D.\btxfnamespacelong G. \btxlastnamefont {Aronson}},
	\btxjtitlefont {\btxifchangecase {Bounds for the fundamental solution of a
			parabolic equation}{Bounds for the fundamental solution of a parabolic
			equation}}, \btxjournalfont {Bull. Amer. Math. Soc.} \btxvolumefont {73}
	(1967), \btxnumbershort {.}~6, 890--896.
	
	\bibitem {kassmann2}
	\btxnamefont {J.~\btxlastnamefont {Bae}} \btxandlong {} \btxnamefont
	{M.~\btxlastnamefont {Kassmann}}, \btxjtitlefont {\btxifchangecase {Schauder
			estimates in generalized {H}\"{o}lder spaces}{Schauder estimates in
			generalized {H}\"{o}lder spaces}}, \btxjournalfont {ArXiv e-prints
		\textbf{1505.05498}} (2015).
	
	\bibitem {imbert}
	\btxnamefont {G.~\btxlastnamefont {Barles}}, \btxnamefont {E.~\btxlastnamefont
		{Chasseigne}}, \btxnamefont {A.~\btxlastnamefont {Ciomaga}}\btxandcomma {}
	\btxandlong {} \btxnamefont {C.~\btxlastnamefont {Imbert}}, \btxjtitlefont
	{\btxifchangecase {Large time behavior of periodic viscosity solutions for
			uniformly parabolic integro-differential equations}{Large time behavior of
			periodic viscosity solutions for uniformly parabolic integro-differential
			equations}}, \btxjournalfont {Calc. Var. Partial Differential Equations}
	\btxvolumefont {50} (2014), \btxnumbershort {.}~1-2, 283--304.
	
	\bibitem {barlow2}
	\btxnamefont {M.\btxfnamespacelong T. \btxlastnamefont {Barlow}}, \btxnamefont
	{R.\btxfnamespacelong F. \btxlastnamefont {Bass}}, \btxnamefont
	{Z.\btxfnamespacelong Q. \btxlastnamefont {Chen}}\btxandcomma {} \btxandlong
	{} \btxnamefont {M.~\btxlastnamefont {Kassmann}}, \btxjtitlefont
	{\btxifchangecase {Non-local {D}irichlet forms and symmetric jump
			processes}{Non-local {D}irichlet forms and symmetric jump processes}},
	\btxjournalfont {Trans. Amer. Math. Soc.} \btxvolumefont {361} (2009),
	\btxnumbershort {.}~4, 1963--1999.
	
	\bibitem {barlow}
	\btxnamefont {M.\btxfnamespacelong T. \btxlastnamefont {Barlow}}, \btxnamefont
	{A.~\btxlastnamefont {Grigor'yan}}\btxandcomma {} \btxandlong {} \btxnamefont
	{T.~\btxlastnamefont {Kumagai}}, \btxjtitlefont {\btxifchangecase {Heat
			kernel upper bounds for jump processes and the first exit time}{Heat kernel
			upper bounds for jump processes and the first exit time}}, \btxjournalfont
	{J. Reine Angew. Math.} \btxvolumefont {626} (2009), 135--157.
	
	\bibitem {benso-lions-book}
	\btxnamefont {A.~\btxlastnamefont {Bensoussan}}, \btxnamefont
	{J\btxfnamespacelong L. \btxlastnamefont {Lions}}\btxandcomma {} \btxandlong
	{} \btxnamefont {G.\btxfnamespacelong C. \btxlastnamefont {Papanicolaou}},
	\btxtitlefont {\btxifchangecase {Asymptotic {A}nalysis for {P}eriodic
			{S}tructures}{Asymptotic {A}nalysis for {P}eriodic {S}tructures}},
	\btxpublisherfont {North-Holland Publishing Co.}, Amsterdam, 1978.
	
	\bibitem {bhat2}
	\btxnamefont {R.\btxfnamespacelong N. \btxlastnamefont {Bhattacharya}},
	\btxjtitlefont {\btxifchangecase {A central limit theorem for diffusions with
			periodic coefficients}{A central limit theorem for diffusions with periodic
			coefficients}}, \btxjournalfont {Ann. Probab.} \btxvolumefont {13} (1985),
	\btxnumbershort {.}~2, 385--396.
	
	\bibitem {goldie}
	\btxnamefont {N.\btxfnamespacelong H. \btxlastnamefont {Bingham}}, \btxnamefont
	{C.\btxfnamespacelong M. \btxlastnamefont {Goldie}}\btxandcomma {}
	\btxandlong {} \btxnamefont {J.\btxfnamespacelong L. \btxlastnamefont
		{Teugels}}, \btxtitlefont {\btxifchangecase {Regular Variation}{Regular
			variation}}, Encyclopedia of Mathematics and its Applications,
	\btxvolumeshort {.}~27, \btxpublisherfont {Cambridge University Press,
		Cambridge}, 1987.
	
	\bibitem {BG-68}
	\btxnamefont {R.\btxfnamespacelong M. \btxlastnamefont {Blumenthal}}
	\btxandlong {} \btxnamefont {R.\btxfnamespacelong K. \btxlastnamefont
		{Getoor}}, \btxtitlefont {\btxifchangecase {Markov Processes and Potential
			Theory}{Markov Processes and Potential Theory}}, \btxpublisherfont {Academic
		Press, New York-London}, 1968.
	
	\bibitem {rene-bjorn-jian}
	\btxnamefont {B.~\btxlastnamefont {B{\"o}ttcher}}, \btxnamefont
	{R.\btxfnamespacelong L. \btxlastnamefont {Schilling}}\btxandcomma {}
	\btxandlong {} \btxnamefont {J.~\btxlastnamefont {Wang}}, \btxtitlefont
	{\btxifchangecase {L\'evy {M}atters. {III}}{L\'evy {M}atters. {III}}},
	\btxpublisherfont {Springer, Cham}, 2013.
	
\bibitem {gosh}
	\btxnamefont {L.~\btxlastnamefont {B\u{a}lilescu}}, \btxnamefont
	{A.~\btxlastnamefont {Ghosh}}\btxandcomma {} \btxandlong {} \btxnamefont
	{T.~\btxlastnamefont {Ghosh}}, \btxjtitlefont {\btxifchangecase
		{Homogenization for non-local elliptic operators in both perforated and
			non-perforated domains}{Homogenization for non-local elliptic operators in
			both perforated and non-perforated domains}}, \btxjournalfont {Zeitschrift f\"{u}r angewandte Mathematik und Physik}
	\btxvolumefont {70} (2019), 
Article number: 171.
	
	\bibitem {chechkin}
	\btxnamefont {G.\btxfnamespacelong A. \btxlastnamefont {Chechkin}},
	\btxnamefont {A.\btxfnamespacelong L. \btxlastnamefont
		{Piatnitski}}\btxandcomma {} \btxandlong {} \btxnamefont
	{A.\btxfnamespacelong S. \btxlastnamefont {Shamaev}}, \btxtitlefont
	{\btxifchangecase {Homogenization}{Homogenization}}, \btxpublisherfont
	{American Mathematical Society, Providence, RI}, 2007.
	
\bibitem {chen-chen-wang}
	\btxnamefont {X.~\btxlastnamefont {Chen}}, \btxnamefont {Z.\btxfnamespacelong
		Q. \btxlastnamefont {Chen}}\btxandcomma {} \btxandlong {} \btxnamefont
	{J.~\btxlastnamefont {Wang}}, \btxjtitlefont {\btxifchangecase {Heat kernel
			for non-local operators with variable order}{Heat kernel for non-local
			operators with variable order}}, \btxjournalfont {Stochastic
		Process. Appl} \btxvolumefont {130} (2020), \btxnumbershort {.}~6, 3574--3647.
	
	\bibitem {chen-hu-xie-zhang}
	\btxnamefont {Z.\btxfnamespacelong Q. \btxlastnamefont {Chen}}, \btxnamefont
	{E.~\btxlastnamefont {Hu}}, \btxnamefont {L.~\btxlastnamefont
		{Xie}}\btxandcomma {} \btxandlong {} \btxnamefont {X.~\btxlastnamefont
		{Zhang}}, \btxjtitlefont {\btxifchangecase {Heat kernels for non-symmetric
			diffusion operators with jumps}{Heat kernels for non-symmetric diffusion
			operators with jumps}}, \btxjournalfont {J. Differential Equations}
	\btxvolumefont {263} (2017), \btxnumbershort {.}~10, 6576--6634.
	

	
	\bibitem {chen-zhang}
	\btxnamefont {Z.\btxfnamespacelong Q. \btxlastnamefont {Chen}} \btxandlong {}
	\btxnamefont {X.~\btxlastnamefont {Zhang}}, \btxjtitlefont {\btxifchangecase
		{Heat kernels and analyticity of non-symmetric jump diffusion
			semigroups}{Heat kernels and analyticity of non-symmetric jump diffusion
			semigroups}}, \btxjournalfont {Probab. Theory Related Fields} \btxvolumefont
	{165} (2016), \btxnumbershort {.}~1-2, 267--312.
	
	\bibitem {chen-zhang2}
	\bysame, \btxjtitlefont {\btxifchangecase {Heat kernels for time-dependent
			non-symmetric stable-like operators}{Heat kernels for time-dependent
			non-symmetric stable-like operators}}, \btxjournalfont {J. Math. Anal. Appl.}
	\btxvolumefont {465} (2018), \btxnumbershort {.}~1, 1--21.
	
	\bibitem {chung}
	\btxnamefont {K.\btxfnamespacelong L. \btxlastnamefont {Chung}} \btxandlong {}
	\btxnamefont {J.\btxfnamespacelong B. \btxlastnamefont {Walsh}},
	\btxtitlefont {\btxifchangecase {Markov Processes, {B}rownian Motion, and
			Time Symmetry}{Markov Processes, {B}rownian Motion, and Time Symmetry}},
	\btxeditionnumshort {second}{.}, \btxpublisherfont {Springer, New York},
	2005.
	
	\bibitem {cinlar}
	\btxnamefont {E.~\btxlastnamefont {{\c{C}}inlar}} \btxandlong {} \btxnamefont
	{J.~\btxlastnamefont {Jacod}}, \btxtitlefont {\btxifchangecase
		{Representation of semimartingale {M}arkov processes in terms of {W}iener
			processes and {P}oisson random measures}{Representation of semimartingale
			{M}arkov processes in terms of {W}iener processes and {P}oisson random
			measures}}, Seminar on {S}tochastic {P}rocesses, \btxpublisherfont
	{Birkh\"auser, Boston, Mass.}, 1981, \btxpagesshort {.}~159--242.
	
	\bibitem {donato}
	\btxnamefont {D.~\btxlastnamefont {Cioranescu}} \btxandlong {} \btxnamefont
	{P.~\btxlastnamefont {Donato}}, \btxtitlefont {\btxifchangecase {An
			Introduction to Homogenization}{An Introduction to Homogenization}},
	\btxpublisherfont {The Clarendon Press, Oxford University Press, New York},
	1999.
	
	\bibitem {courrege-symbol}
	\btxnamefont {Ph. \btxlastnamefont {Courr\'{e}ge}}, \btxjtitlefont
	{\btxifchangecase {Sur la forme int\'{e}gro-diff\'{e}rentielle des
			op\'{e}rateurs de ${C}^{\infty}_{K}$ dans ${C}$ satisfaisant au principe du
			maximum}{Sur la forme int\'{e}gro-diff\'{e}rentielle des op\'{e}rateurs de
			${C}^{\infty}_{K}$ dans ${C}$ satisfaisant au principe du maximum}},
	\btxjournalfont {S\'{e}m. Th\'{e}orie du Potentiel} \btxvolumefont
	{expos\'{e} 2} (1965-1966), 38 pp.
	
	\bibitem {doob}
	\btxnamefont {J.\btxfnamespacelong L. \btxlastnamefont {Doob}}, \btxtitlefont
	{\btxifchangecase {Stochastic {P}rocesses}{Stochastic {P}rocesses}},
	\btxpublisherfont {John Wiley \& Sons, Inc., New York; Chapman \& Hall,
		Limited, London}, 1953.
	
	\bibitem {durrett}
	\btxnamefont {R.~\btxlastnamefont {Durrett}}, \btxtitlefont {\btxifchangecase
		{Stochastic Calculus}{Stochastic Calculus}}, \btxpublisherfont {CRC Press,
		Boca Raton, FL}, 1996.
	
	\bibitem {errami}
	\btxnamefont {M.~\btxlastnamefont {Errami}}, \btxnamefont {F.~\btxlastnamefont
		{Russo}}\btxandcomma {} \btxandlong {} \btxnamefont {P.~\btxlastnamefont
		{Vallois}}, \btxjtitlefont {\btxifchangecase {It\^{o}'s formula for
			{$C^{1,\lambda}$}-functions of a c\`adl\`ag process and related
			calculus}{It\^{o}'s formula for {$C^{1,\lambda}$}-functions of a c\`adl\`ag
			process and related calculus}}, \btxjournalfont {Probab. Theory Related
		Fields} \btxvolumefont {122} (2002), \btxnumbershort {.}~2, 191--221.
	
	\bibitem {ethier}
	\btxnamefont {S.\btxfnamespacelong N. \btxlastnamefont {Ethier}} \btxandlong {}
	\btxnamefont {T.\btxfnamespacelong G. \btxlastnamefont {Kurtz}},
	\btxtitlefont {\btxifchangecase {Markov {P}rocesses}{Markov {P}rocesses}},
	\btxpublisherfont {John Wiley \& Sons Inc.}, New York, 1986.
	
	\bibitem {salort}
	\btxnamefont {J.~\btxlastnamefont {Fern\'{a}ndez~Bonder}}, \btxnamefont
	{A.~\btxlastnamefont {Ritorto}}\btxandcomma {} \btxandlong {} \btxnamefont
	{A.\btxfnamespacelong M. \btxlastnamefont {Salort}}, \btxjtitlefont
	{\btxifchangecase {{$H$}-convergence result for nonlocal elliptic-type
			problems via {T}artar's method}{{$H$}-convergence result for nonlocal
			elliptic-type problems via {T}artar's method}}, \btxjournalfont {SIAM J.
		Math. Anal.} \btxvolumefont {49} (2017), \btxnumbershort {.}~4, 2387--2408.
	
	\bibitem {focardi}
	\btxnamefont {M.~\btxlastnamefont {Focardi}}, \btxjtitlefont {\btxifchangecase
		{Aperiodic fractional obstacle problems}{Aperiodic fractional obstacle
			problems}}, \btxjournalfont {Adv. Math.} \btxvolumefont {225} (2010),
	\btxnumbershort {.}~6, 3502--3544.
	
	\bibitem {franke-periodic}
	\btxnamefont {B.~\btxlastnamefont {Franke}}, \btxjtitlefont {\btxifchangecase
		{The scaling limit behaviour of periodic stable-like processes}{The scaling
			limit behaviour of periodic stable-like processes}}, \btxjournalfont
	{Bernoulli} \btxvolumefont {12} (2006), \btxnumbershort {.}~3, 551--570.
	
	\bibitem {franke-periodicerata}
	\bysame, \btxjtitlefont {\btxifchangecase {Correction to: ``{T}he scaling limit
			behaviour of periodic stable-like processes'' [{B}ernoulli {\bf 12} (2006),
			no. 3, 551--570]}{Correction to: ``{T}he scaling limit behaviour of periodic
			stable-like processes'' [{B}ernoulli {\bf 12} (2006), no. 3, 551--570]}},
	\btxjournalfont {Bernoulli} \btxvolumefont {13} (2007), \btxnumbershort
	{.}~2, 600.
	
	\bibitem {franke2}
	\bysame, \btxjtitlefont {\btxifchangecase {A functional non-central limit
			theorem for jump-diffusions with periodic coefficients driven by stable
			{L}\'{e}vy-noise}{A functional non-central limit theorem for jump-diffusions
			with periodic coefficients driven by stable {L}\'{e}vy-noise}},
	\btxjournalfont {J. Theoret. Probab.} \btxvolumefont {20} (2007),
	\btxnumbershort {.}~4, 1087--1100.
	
	
	

	
	
	\bibitem {fried}
	\btxnamefont {A.~\btxlastnamefont {Friedman}}, \btxtitlefont {\btxifchangecase
		{Stochastic {D}ifferential {E}quations and {A}pplications. {V}ol.
			1}{Stochastic {D}ifferential {E}quations and {A}pplications. {V}ol. 1}},
	\btxpublisherfont {Academic Press [Harcourt Brace Jovanovich, Publishers],
		New York-London}, 1975.
	
	\bibitem {fujiwara}
	\btxnamefont {T.~\btxlastnamefont {Fujiwara}} \btxandlong {} \btxnamefont
	{M.~\btxlastnamefont {Tomisaki}}, \btxjtitlefont {\btxifchangecase
		{Martingale approach to limit theorems for jump processes}{Martingale
			approach to limit theorems for jump processes}}, \btxjournalfont {Stochastics
		Stochastics Rep.} \btxvolumefont {50} (1994), \btxnumbershort {.}~1-2,
	35--64.
	
	\bibitem {grz-sz}
	\btxnamefont {T.~\btxlastnamefont {Grzywny}} \btxandlong {} \btxnamefont
	{K.~\btxlastnamefont {Szczypkowski}}, \btxjtitlefont {\btxifchangecase {Heat
			kernels of non-symmetric {L}\'{e}vy-type operators}{Heat kernels of
			non-symmetric {L}\'{e}vy-type operators}}, \btxjournalfont {J. Differential
		Equations} \btxvolumefont {267} (2019), \btxnumbershort {.}~10, 6004--6064.
	
	\bibitem {horie-inuzuka-tanaka}
	\btxnamefont {M.~\btxlastnamefont {Horie}}, \btxnamefont {T.~\btxlastnamefont
		{Inuzuka}}\btxandcomma {} \btxandlong {} \btxnamefont {H.~\btxlastnamefont
		{Tanaka}}, \btxjtitlefont {\btxifchangecase {Homogenization of certain
			one-dimensional discontinuous {M}arkov processes}{Homogenization of certain
			one-dimensional discontinuous {M}arkov processes}}, \btxjournalfont
	{Hiroshima Math. J.} \btxvolumefont {7} (1977), \btxnumbershort {.}~2,
	629--641.
	
	\bibitem {huang-duan-song}
	\btxnamefont {Q.~\btxlastnamefont {Huang}}, \btxnamefont {J.~\btxlastnamefont
		{Duan}}\btxandcomma {} \btxandlong {} \btxnamefont {R.~\btxlastnamefont
		{Song}}, \btxjtitlefont {\btxifchangecase {Homogenization of nonlocal
			partial differential equations related to stochastic differential
			equations with {L}\'evy noise}{Homogenization of nonlocal partial
			differential equations related to stochastic differential
			equations with {L}\'evy noise}}, \btxjournalfont {ArXiv e-prints
		\textbf{1804.06555}} (2018).
	
	\bibitem {huang-duan-song2}
	\bysame, \btxjtitlefont {\btxifchangecase {Homogenization of stable-like
			{F}eller processes}{Homogenization of stable-like {F}eller
			processes}}, \btxjournalfont {ArXiv e-prints \textbf{1812.11624}} (2018).
	
	\bibitem {ikeda-watanabe}
	\btxnamefont {N.~\btxlastnamefont {Ikeda}} \btxandlong {} \btxnamefont
	{S.~\btxlastnamefont {Watanabe}}, \btxtitlefont {\btxifchangecase {Stochastic
			Differential Equations and Diffusion Processes}{Stochastic Differential
			Equations and Diffusion Processes}}, \btxeditionnumshort {second}{.},
	\btxpublisherfont {North-Holland Publishing Co., Amsterdam; Kodansha, Ltd.,
		Tokyo}, 1989.
	
	\bibitem {jacobI}
	\btxnamefont {N.~\btxlastnamefont {Jacob}}, \btxtitlefont {\btxifchangecase
		{Pseudo {D}ifferential {O}perators and {M}arkov {P}rocesses. {V}ol.
			{I}}{Pseudo {D}ifferential {O}perators and {M}arkov {P}rocesses. {V}ol.
			{I}}}, \btxpublisherfont {Imperial College Press}, London, 2001.
	
	\bibitem {jacobIII}
	\bysame, \btxtitlefont {\btxifchangecase {Pseudo {D}ifferential {O}perators and
			{M}arkov {P}rocesses. {V}ol. {III}}{Pseudo {D}ifferential {O}perators and
			{M}arkov {P}rocesses. {V}ol. {III}}}, \btxpublisherfont {Imperial College
		Press}, London, 2005.
	
	\bibitem {jacod}
	\btxnamefont {J.~\btxlastnamefont {Jacod}} \btxandlong {} \btxnamefont
	{A.\btxfnamespacelong N. \btxlastnamefont {Shiryaev}}, \btxtitlefont
	{\btxifchangecase {Limit {T}heorems for {S}tochastic {P}rocesses}{Limit
			{T}heorems for {S}tochastic {P}rocesses}}, \btxeditionnumshort {second}{.},
	\btxvolumeshort {.}\ 288, \btxpublisherfont {Springer-Verlag}, Berlin, 2003.
	
	\bibitem {kozlov}
	\btxnamefont {V.\btxfnamespacelong V. \btxlastnamefont {Jikov}}, \btxnamefont
	{S.\btxfnamespacelong M. \btxlastnamefont {Kozlov}}\btxandcomma {}
	\btxandlong {} \btxnamefont {O.\btxfnamespacelong A. \btxlastnamefont
		{Ole{\u\i}nik}}, \btxtitlefont {\btxifchangecase {Homogenization of
			{D}ifferential {O}perators and {I}ntegral {F}unctionals}{Homogenization of
			{D}ifferential {O}perators and {I}ntegral {F}unctionals}}, \btxpublisherfont
	{Springer-Verlag, Berlin}, 1994.
	
	\bibitem {kassmann}
	\btxnamefont {M.~\btxlastnamefont {Kassmann}}, \btxnamefont
	{A.~\btxlastnamefont {Piatnitski}}\btxandcomma {} \btxandlong {} \btxnamefont
	{E.~\btxlastnamefont {Zhizhina}}, \btxjtitlefont {\btxifchangecase
		{Homogenization of {L}\'{e}vy-type operators with oscillating
			coefficients}{Homogenization of {L}\'{e}vy-type operators with oscillating
			coefficients}}, \btxjournalfont {SIAM J. Math. Anal.} \btxvolumefont {51}
	(2019), \btxnumbershort {.}~5, 3641--3665.
	
	\bibitem {kim-lee}
	\btxnamefont {P.~\btxlastnamefont {Kim}} \btxandlong {} \btxnamefont
	{J.~\btxlastnamefont {Lee}}, \btxjtitlefont {\btxifchangecase {Heat kernels
			of non-symmetric jump processes with exponentially decaying jumping
			kernel}{Heat kernels of non-symmetric jump processes with exponentially
			decaying jumping kernel}}, \btxjournalfont {Stochastic Process. Appl.}
	\btxvolumefont {129} (2019), \btxnumbershort {.}~6, 2130--2173.
	
	\bibitem {kim-sig-vond}
	\btxnamefont {P.~\btxlastnamefont {Kim}}, \btxnamefont {R.~\btxlastnamefont
		{Song}}\btxandcomma {} \btxandlong {} \btxnamefont {Z.~\btxlastnamefont
		{Vondra\v{c}ek}}, \btxjtitlefont {\btxifchangecase {Heat kernels of
			non-symmetric jump processes: beyond the stable case}{Heat kernels of
			non-symmetric jump processes: beyond the stable case}}, \btxjournalfont
	{Potential Anal.} \btxvolumefont {49} (2018), \btxnumbershort {.}~1, 37--90.
	
	\bibitem {kno-sci}
	\btxnamefont {V.~\btxlastnamefont {Knopova}} \btxandlong {} \btxnamefont
	{R.\btxfnamespacelong L. \btxlastnamefont {Schilling}}, \btxjtitlefont
	{\btxifchangecase {A note on the existence of transition probability
			densities of {L}\'{e}vy processes}{A note on the existence of transition
			probability densities of {L}\'{e}vy processes}}, \btxjournalfont {Forum
		Math.} \btxvolumefont {25} (2013), \btxnumbershort {.}~1, 125--149.
	
	\bibitem {vasili-book}
	\btxnamefont {V.\btxfnamespacelong N. \btxlastnamefont {Kolokoltsov}},
	\btxtitlefont {\btxifchangecase {Markov {P}rocesses, {S}emigroups and
			{G}enerators}{Markov {P}rocesses, {S}emigroups and {G}enerators}},
	\btxvolumeshort {.}~38, \btxpublisherfont {Walter de Gruyter \& Co.}, Berlin,
	2011.
	
	\bibitem {kwon}
	\btxnamefont {Y.~\btxlastnamefont {Kwon}} \btxandlong {} \btxnamefont
	{C.~\btxlastnamefont {Lee}}, \btxjtitlefont {\btxifchangecase {Strong
			{F}eller property and irreducibility of diffusions with jumps}{Strong
			{F}eller property and irreducibility of diffusions with jumps}},
	\btxjournalfont {Stochastics Stochastics Rep.} \btxvolumefont {67} (1999),
	\btxnumbershort {.}~1-2, 147--157.
	
	\bibitem {liang-sch-wang}
	\btxnamefont {M.~\btxlastnamefont {Liang}}, \btxnamefont {R.\btxfnamespacelong
		L. \btxlastnamefont {Schilling}}\btxandcomma {} \btxandlong {} \btxnamefont
	{J.~\btxlastnamefont {Wang}}, \btxjtitlefont {\btxifchangecase {A unified
			approach to coupling {S}{D}{E}s driven by {L}\'evy noise and some
			applications}{A unified approach to coupling {S}{D}{E}s driven by {L}\'evy
			noise and some applications}}, \btxjournalfont {Bernoulli} \btxvolumefont {26} (2020),
	\btxnumbershort {.}~1, 664--693.
	
\bibitem {liang-wang}
	\btxnamefont {M.~\btxlastnamefont {Liang}} \btxandlong {} \btxnamefont
	{J.~\btxlastnamefont {Wang}}, \btxjtitlefont {\btxifchangecase {Gradient
			estimates and ergodicity for {S}{D}{E}s driven by multiplicative {L}\'{e}vy
			noises via coupling}{Gradient estimates and ergodicity for {S}{D}{E}s driven
			by multiplicative {L}\'{e}vy noises via coupling}}, \btxjournalfont {Stochastic Process. Appl.}
\btxvolumefont {130} (2020), \btxnumbershort {.}~5, 3053--3094.
	
	\bibitem {liang-wang2}
	\bysame, \btxjtitlefont {\btxifchangecase {Spatial regularity of semigroups
			generated by {L}\'{e}vy type operators}{Spatial regularity of semigroups
			generated by {L}\'{e}vy type operators}}, \btxjournalfont {Math. Nachr.}
	\btxvolumefont {292} (2019), \btxnumbershort {.}~7, 1551--1566.
	
	\bibitem {luo-wang}
	\btxnamefont {D.~\btxlastnamefont {Luo}} \btxandlong {} \btxnamefont
	{J.~\btxlastnamefont {Wang}}, \btxjtitlefont {\btxifchangecase {Refined basic
			couplings and {W}asserstein-type distances for {SDE}s with {L}\'{e}vy
			noises}{Refined basic couplings and {W}asserstein-type distances for {SDE}s
			with {L}\'{e}vy noises}}, \btxjournalfont {Stochastic Process. Appl.}
	\btxvolumefont {129} (2019), \btxnumbershort {.}~9, 3129--3173.
	
	\bibitem {masuda}
	\btxnamefont {H.~\btxlastnamefont {Masuda}}, \btxjtitlefont {\btxifchangecase
		{Ergodicity and exponential {$\beta$}-mixing bounds for multidimensional
			diffusions with jumps}{Ergodicity and exponential {$\beta$}-mixing bounds for
			multidimensional diffusions with jumps}}, \btxjournalfont {Stochastic
		Process. Appl.} \btxvolumefont {117} (2007), \btxnumbershort {.}~1, 35--56.
	
	\bibitem {masuda-err}
	\bysame, \btxjtitlefont {\btxifchangecase {Erratum to: ``{E}rgodicity and
			exponential {$\beta$}-mixing bound for multidimensional diffusions with
			jumps'' [{S}tochastic {P}rocess. {A}ppl. 117 (2007) 35--56]
		}{Erratum to: ``{E}rgodicity and exponential {$\beta$}-mixing
			bound for multidimensional diffusions with jumps'' [{S}tochastic {P}rocess.
			{A}ppl. 117 (2007) 35--56]}}, \btxjournalfont {Stochastic
		Process. Appl.} \btxvolumefont {119} (2009), \btxnumbershort {.}~2, 676--678.
	
	\bibitem {meyn-tweedie-II}
	\btxnamefont {S.\btxfnamespacelong P. \btxlastnamefont {Meyn}} \btxandlong {}
	\btxnamefont {R.\btxfnamespacelong L. \btxlastnamefont {Tweedie}},
	\btxjtitlefont {\btxifchangecase {Stability of {M}arkovian {M}rocesses. {II}.
			{C}ontinuous-time processes and sampled chains}{Stability of {M}arkovian
			{M}rocesses. {II}. {C}ontinuous-time processes and sampled chains}},
	\btxjournalfont {Adv. Appl. Probab.} \btxvolumefont {25} (1993),
	\btxnumbershort {.}~3, 487--517.
	
	\bibitem {meynIII}
	\bysame, \btxjtitlefont {\btxifchangecase {Stability of {M}arkovian processes.
			{III}. {F}oster-{L}yapunov criteria for continuous-time processes}{Stability
			of {M}arkovian processes. {III}. {F}oster-{L}yapunov criteria for
			continuous-time processes}}, \btxjournalfont {Adv. Appl. Probab.}
	\btxvolumefont {25} (1993), \btxnumbershort {.}~3, 518--548.
	
	\bibitem {nash}
	\btxnamefont {J.~\btxlastnamefont {Nash}}, \btxjtitlefont {\btxifchangecase
		{Continuity of solutions of parabolic and elliptic equations}{Continuity of
			solutions of parabolic and elliptic equations}}, \btxjournalfont {Amer. J.
		Math.} \btxvolumefont {80} (1958), \btxnumbershort {.}~4, 931--954.
	
	\bibitem {tomisaki}
	\btxnamefont {Y.~\btxlastnamefont {Ogura}}, \btxnamefont {M.~\btxlastnamefont
		{Tomisaki}}\btxandcomma {} \btxandlong {} \btxnamefont {M.~\btxlastnamefont
		{Tsuchiya}}, \btxtitlefont {\btxifchangecase {Existence of a strong solution
			for an integro-differential equation and superposition of diffusion
			processes}{Existence of a strong solution for an integro-differential
			equation and superposition of diffusion processes}}, Stochastics in finite
	and infinite dimensions, Trends Math., \btxpublisherfont {Birkh\"{a}user
		Boston, Boston, MA}, 2001, \btxpagesshort {.}~341--359.
	
	\bibitem {pang-sand}
	\btxnamefont {G.~\btxlastnamefont {Pang}} \btxandlong {} \btxnamefont
	{N.~\btxlastnamefont {Sandri\'{c}}}, \btxjtitlefont {\btxifchangecase
		{Ergodicity and fluctuations of a fluid particle driven by diffusions with
			jumps}{Ergodicity and fluctuations of a fluid particle driven by diffusions
			with jumps}}, \btxjournalfont {Commun. Math. Sci.} \btxvolumefont {14}
	(2016), \btxnumbershort {.}~2, 327--362.
	
	\bibitem {zhizhina}
	\btxnamefont {A.~\btxlastnamefont {Piatnitski}} \btxandlong {} \btxnamefont
	{E.~\btxlastnamefont {Zhizhina}}, \btxjtitlefont {\btxifchangecase {Periodic
			homogenization of nonlocal operators with a convolution-type kernel}{Periodic
			homogenization of nonlocal operators with a convolution-type kernel}},
	\btxjournalfont {SIAM J. Math. Anal.} \btxvolumefont {49} (2017),
	\btxnumbershort {.}~1, 64--81.
	
	\bibitem {priola}
	\btxnamefont {E.~\btxlastnamefont {Priola}}, \btxjtitlefont {\btxifchangecase
		{Pathwise uniqueness for singular {SDE}s driven by stable processes}{Pathwise
			uniqueness for singular {SDE}s driven by stable processes}}, \btxjournalfont
	{Osaka J. Math.} \btxvolumefont {49} (2012), \btxnumbershort {.}~2, 421--447.
	
	\bibitem {rhodes}
	\btxnamefont {R.~\btxlastnamefont {Rhodes}} \btxandlong {} \btxnamefont
	{V.~\btxlastnamefont {Vargas}}, \btxjtitlefont {\btxifchangecase {Scaling
			limits for symmetric {I}t\^o-{L}\'evy processes in random medium}{Scaling
			limits for symmetric {I}t\^o-{L}\'evy processes in random medium}},
	\btxjournalfont {Stochastic Process. Appl.} \btxvolumefont {119} (2009),
	\btxnumbershort {.}~12, 4004--4033.
	
	\bibitem {rogersI}
	\btxnamefont {L.\btxfnamespacelong C.\btxfnamespacelong G. \btxlastnamefont
		{Rogers}} \btxandlong {} \btxnamefont {D.~\btxlastnamefont {Williams}},
	\btxtitlefont {\btxifchangecase {Diffusions, {M}arkov {P}rocesses, and
			{M}artingales. {V}ol. 1}{Diffusions, {M}arkov {P}rocesses, and {M}artingales.
			{V}ol. 1}}, \btxpublisherfont {Cambridge University Press}, Cambridge, 2000.
	
	\bibitem {rogersII}
	\bysame, \btxtitlefont {\btxifchangecase {Diffusions, {M}arkov {P}rocesses, and
			{M}artingales. {V}ol. 2}{Diffusions, {M}arkov {P}rocesses, and {M}artingales.
			{V}ol. 2}}, \btxpublisherfont {Cambridge University Press}, Cambridge, 2000.
	
	\bibitem {hom}
	\btxnamefont {N.~\btxlastnamefont {Sandri\'{c}}}, \btxjtitlefont
	{\btxifchangecase {Homogenization of periodic diffusion with small
			jumps}{Homogenization of periodic diffusion with small jumps}},
	\btxjournalfont {J. Math. Anal. Appl.} \btxvolumefont {435} (2016),
	\btxnumbershort {.}~1, 551--577.
	
	\bibitem {sato-book}
	\btxnamefont {K.~\btxlastnamefont {Sato}}, \btxtitlefont {\btxifchangecase
		{L\'evy {P}rocesses and {I}nfinitely {D}ivisible {D}istributions}{L\'evy
			{P}rocesses and {I}nfinitely {D}ivisible {D}istributions}}, \btxpublisherfont
	{Cambridge University Press}, Cambridge, 1999.
	
	\bibitem {rene-conserv}
	\btxnamefont {R.\btxfnamespacelong L. \btxlastnamefont {Schilling}},
	\btxjtitlefont {\btxifchangecase {Conservativeness and extensions of Feller
			semigroups}{Conservativeness and extensions of Feller semigroups}},
	\btxjournalfont {Positivity} \btxvolumefont {2} (1998), 239--256.
	
	\bibitem {rene-holder}
	\bysame, \btxjtitlefont {\btxifchangecase {Growth and {H}\"older conditions for
			the sample paths of {F}eller processes}{Growth and {H}\"older conditions for
			the sample paths of {F}eller processes}}, \btxjournalfont {Probab. Theory
		Related Fields} \btxvolumefont {112} (1998), \btxnumbershort {.}~4, 565--611.
	
	\bibitem {schnurr}
	\btxnamefont {R.\btxfnamespacelong L. \btxlastnamefont {Schilling}} \btxandlong
	{} \btxnamefont {A.~\btxlastnamefont {Schnurr}}, \btxjtitlefont
	{\btxifchangecase {The symbol associated with the solution of a stochastic
			differential equation}{The symbol associated with the solution of a
			stochastic differential equation}}, \btxjournalfont {Electron. J. Probab.}
	\btxvolumefont {15} (2010), 1369--1393.
	
	\bibitem {jian}
	\btxnamefont {R.\btxfnamespacelong L. \btxlastnamefont {Schilling}},
	\btxnamefont {P.~\btxlastnamefont {Sztonyk}}\btxandcomma {} \btxandlong {}
	\btxnamefont {J.~\btxlastnamefont {Wang}}, \btxjtitlefont {\btxifchangecase
		{Coupling property and gradient estimates of {L}\'{e}vy processes via the
			symbol}{Coupling property and gradient estimates of {L}\'{e}vy processes via
			the symbol}}, \btxjournalfont {Bernoulli} \btxvolumefont {18} (2012),
	\btxnumbershort {.}~4, 1128--1149.
	
\bibitem {uemura}
	\btxnamefont {R.\btxfnamespacelong L. \btxlastnamefont {Schilling}} \btxandlong
	{} \btxnamefont {T.~\btxlastnamefont {Uemura}}, \btxjtitlefont
	{\btxifchangecase {Homogenization of symmetric {L}\'evy processes on
			$\mathbb{R}^d$}{Homogenization of symmetric {L}\'evy processes on
			$\mathbb{R}^d$}}, \btxjournalfont {to appear in Revue Roumaine de Math\'ematiques Pures et Appliqu\'ees, ArXiv e-prints \textbf{1808.01667}}
	(2018).
	
	\bibitem {strong}
	\btxnamefont {R.\btxfnamespacelong L. \btxlastnamefont {Schilling}} \btxandlong
	{} \btxnamefont {J.~\btxlastnamefont {Wang}}, \btxjtitlefont
	{\btxifchangecase {Strong {F}eller continuity of {F}eller processes and
			semigroups}{Strong {F}eller continuity of {F}eller processes and
			semigroups}}, \btxjournalfont {Infin. Dimens. Anal. Quantum Probab. Relat.
		Top.} \btxvolumefont {15} (2012), \btxnumbershort {.}~2, 1250010, 28.
	
	\bibitem {schwabI}
	\btxnamefont {R.\btxfnamespacelong W. \btxlastnamefont {Schwab}},
	\btxjtitlefont {\btxifchangecase {Periodic homogenization for nonlinear
			integro-differential equations}{Periodic homogenization for nonlinear
			integro-differential equations}}, \btxjournalfont {SIAM J. Math. Anal.}
	\btxvolumefont {42} (2010), \btxnumbershort {.}~6, 2652--2680.
	
	\bibitem {schwabII}
	\bysame, \btxjtitlefont {\btxifchangecase {Stochastic homogenization for some
			nonlinear integro-differential equations}{Stochastic homogenization for some
			nonlinear integro-differential equations}}, \btxjournalfont {Comm. Partial
		Differential Equations} \btxvolumefont {38} (2013), \btxnumbershort {.}~2,
	171--198.
	
	\bibitem {shioz}
	\btxnamefont {Y.~\btxlastnamefont {Shiozawa}} \btxandlong {} \btxnamefont
	{T.~\btxlastnamefont {Uemura}}, \btxjtitlefont {\btxifchangecase {Stability
			of the {F}eller property for non-local operators under bounded
			perturbations}{Stability of the {F}eller property for non-local operators
			under bounded perturbations}}, \btxjournalfont {Glas. Mat. Ser. III}
	\btxvolumefont {45} (2010), \btxnumbershort {.}~1, 155--172.
	
	\bibitem {simon}
	\btxnamefont {T.~\btxlastnamefont {Simon}}, \btxjtitlefont {\btxifchangecase
		{Support theorem for jump processes}{Support theorem for jump processes}},
	\btxjournalfont {Stochastic Process. Appl.} \btxvolumefont {89} (2000),
	\btxnumbershort {.}~1, 1--30.
	
	\bibitem {str}
	\btxnamefont {D.\btxfnamespacelong W. \btxlastnamefont {Stroock}},
	\btxjtitlefont {\btxifchangecase {Diffusion processes associated with
			{L}\'{e}vy generators}{Diffusion processes associated with {L}\'{e}vy
			generators}}, \btxjournalfont {Z. Wahrscheinlichkeitstheorie und Verw.
		Gebiete} \btxvolumefont {32} (1975), \btxnumbershort {.}~3, 209--244.
	
	\bibitem {stroock}
	\bysame, \btxtitlefont {\btxifchangecase {Diffusion semigroups corresponding to
			uniformly elliptic divergence form operators}{Diffusion semigroups
			corresponding to uniformly elliptic divergence form operators}},
	S\'{e}minaire de {P}robabilit\'{e}s, {XXII}, Lecture Notes in Math.,
	\btxvolumeshort {.}\ 1321, \btxpublisherfont {Springer, Berlin}, 1988,
	\btxpagesshort {.}~316--347.
	
	\bibitem {tartarII}
	\btxnamefont {L.\btxfnamespacelong C. \btxlastnamefont {Tartar}}, \btxtitlefont
	{\btxifchangecase {The {G}eneral {T}heory of {H}omogenization}{The {G}eneral
			{T}heory of {H}omogenization}}, \btxpublisherfont {Springer-Verlag, Berlin;
		UMI, Bologna}, 2009.
	
	
		\bibitem {Tom}
	\btxnamefont {M.~\btxlastnamefont {Tomisaki}}, \btxjtitlefont {\btxifchangecase
		{Homogenization of c\`adl\`ag processes}{Homogenization of c\`adl\`ag
			processes}}, \btxjournalfont {J. Math. Soc. Japan} \btxvolumefont {44}
	(1992), \btxnumbershort {.}~2, 281--305.
	
	\bibitem {tweedie-mproc}
	\btxnamefont {R.\btxfnamespacelong L. \btxlastnamefont {Tweedie}},
	\btxjtitlefont {\btxifchangecase {Topological conditions enabling use of
			{H}arris methods in discrete and continuous time}{Topological conditions
			enabling use of {H}arris methods in discrete and continuous time}},
	\btxjournalfont {Acta Appl. Math.} \btxvolumefont {34} (1994),
	\btxnumbershort {.}~1-2, 175--188.
	
	\bibitem{WZ}
	L.~Wang and X.~Zhang.
\btxjtitlefont {\btxifchangecase
		{Harnack inequalities for SDEs driven by cylindrical
	$\alpha$-stable processes}},
	\btxjournalfont {Potential Anal} \btxvolumefont {42} (2015), 657--669,
	2015.
	
	
	\bibitem {wat}
	\btxnamefont {H.~\btxlastnamefont {Watanabe}}, \btxjtitlefont {\btxifchangecase
		{Potential operator of a recurrent strong {F}eller process in the strict
			sense and boundary value problem}{Potential operator of a recurrent strong
			{F}eller process in the strict sense and boundary value problem}},
	\btxjournalfont {J. Math. Soc. Japan} \btxvolumefont {16} (1964), 	\btxnumbershort {.}~2, 83--95.
	
\end{thebibliography}
\end{document}